\documentclass[preprint,3p,12pt,times]{elsarticle}
\usepackage{amsmath,amsthm,subfigure,latexsym,amsfonts,amssymb,mathrsfs}
\usepackage{graphicx,verbatim}
\usepackage{graphicx}

\usepackage{multicol,amsmath,amssymb,amstext,graphicx}
\usepackage{color}
\newcommand{\be}{\begin{eqnarray*}}                           
\newcommand{\en}{\end{eqnarray*}}
\newcommand{\bes}{\begin{eqnarray}}                           
\newcommand{\ens}{\end{eqnarray}}

\def\nn{\nonumber}

\newcommand{\la}{\lambda}

\newcommand{\ep}{\epsilon}

\newcommand{\coshe}{ \cosh^\ep(\canla  t)}
\newcommand{\sinhe}{ \sinh^\ep(\canla t)}
\newcommand{\sinhes}{ \sinh^\ep\left(\canla (t-s)\right)}
\newcommand{\coshh}{ \cosh \left( \sqrt{\lambda_p}t\right)}
\newcommand{\sinhh}{ \sinh \left( \sqrt{\lambda_p}t\right)}
\newcommand{\coshs}{ \cosh \left( \sqrt{\lambda_p}(t-s)\right)}
\newcommand{\sinhs}{ \sinh \left( \sqrt{\lambda_p}(t-s)\right)}
\newcommand{\canla}{  \sqrt{\lambda_p} }

\newcommand{\R}{\mathbb{R}  }

\newtheorem{theorem}{Theorem}[section]

\newtheorem{lemma}{Lemma}[section]

\newtheorem{remark}{Remark}[section]

\def\bq{\begin{equation}}
\def\eq{\end{equation}}
\def\bqq{\begin{eqnarray*}}
\def\eqq{\end{eqnarray*}}

\def\nn{\nonumber}

\usepackage{multicol,amsmath,amstext,graphics,graphicx}
\usepackage{fancyhdr}

\begin{document}

\begin{frontmatter}
\title{ A modified integral equation method of the nonlinear elliptic equation with globally and locally Lipschitz source}

\author{Nguyen Huy Tuan$^1$,~~ Le Duc Thang $^2$,~~ Vo Anh Khoa$^{1,3}$  \\
\small $^{1}$Department of Mathematics and Computer Science, Ho Chi Minh City University of Science,\\
\small  227 Nguyen Van Cu, District 5, Ho Chi Minh City, Vietnam.\\
\small $^{2}$Faculty of Basic Science, Ho Chi Minh City Industry and Trade College,\\
\small 20 Tang Nhon Phu, District 9, Ho Chi Minh City, Vietnam.\\
\small $^{3}$Mathematics and Computer Science Division, Gran Sasso Science Institute, Viale Francesco Crispi 7, 67100 L'Aquila, Italy.\\
}

\date{\today}

\begin{abstract}
{\textcolor{red}{The paper is devoted to investigating} a Cauchy problem for  nonlinear elliptic \textcolor{red}{PDEs}   in the \textcolor{red}{abstract} Hilbert space. \textcolor{red}{The problem is hardly solved by computation since it} is severely ill-posed in the sense of Hadamard.  We shall use a modified integral equation
method to regularize the nonlinear problem with globally and locally Lipschitz source terms.   Convergence estimates are established under priori assumptions on exact solution.    A numerical test \textcolor{red}{is provided to illustrate} that the proposed method is feasible and effective.  These results extend some earlier works on a Cauchy problem for elliptic equations}
\end{abstract}

\end{frontmatter}

\section{Introduction}
Let $H$ be a Hilbert space with the inner product $\left\langle \cdot,\cdot\right\rangle $ and the norm $\|\cdot\|$, and let $A : D(A)\subset 
H \to  H$ be a positive-definite, self-adjoint operator with \textcolor{red}{a} compact inverse on $H$.  Let $T$ be a positive number, we consider the \textcolor{red}{general} problem of finding a function $u: [0,T] \to H $ from the system 
\bq 
\left\{ \begin{gathered}
 u_{tt}=Au+f(t,u(t)), t \in   (0,T), \hfill \\
  u(0)=\varphi,  \hfill \\
 u_t(0)=g ,   \hfill\\
 \end{gathered}  \right.\label{e1}
\eq
where the  \textcolor{red}{datum}   $g, \varphi$ are given  in  $ H$ and the source function $f : [0,T] \times  H \to H$ is defined later. \\

 The problem \eqref{e1} is a generalization of \textcolor{red}{several} well-known equations. For a simple example, \textcolor{red}{taking $A=-\frac{\partial^2}{\partial x^2}$  and  $ f(x,t,u)=-k^2 u$  gives us the Helmholtz equation arising} in many engineering applications
related to propagating waves in different environments, such as acoustic, hydrodynamic
and electromagnetic waves (see, e.g. \cite{Be,Chen}). \textcolor{red}{Furthermore, if we consider the linear case, i.e. the source term does not rely on $u$, and $A$ the
second-order differential operator defined in $H_0^1(\Omega)$
\be
Au = -\sum\limits_{i,j = 1}^n {\frac{\partial }{{\partial {x_i}}}({a_{ij}}(x)\frac{\partial u }{{\partial {x_j}}}) + a(x)u},
\en
with $a_{ij}, a \in L^\infty(\Omega)$ satisfying %
\[\sum\limits_{i,j = 1}^n {{a_{ij}}{\xi _i}{\xi _j} \le \nu \sum\limits_{i = 1}^n {|{\xi _i}{|^2}} }, \]
where $\nu>0$ is given,  $a(x) \ge 0$, then we obtain the Poisson equation that obviously has been studied in a very long time in both pure and applied aspects, e.g. \cite{D+N+V+2009}.}\\
  
  \textcolor{red}{Consequently, the inverse problems governed by elliptic PDEs play important roles  in many physical and engineering problems. However, there always exists a big question concerning the fact that how to deal with the instability. In fact, such a problem like \eqref{e1} is not well-posed in the sense that a small perturbation in the given Cauchy data $(\varphi,g)$ may effect a very large error on the solution.}
Therefore, it is very difficult to solve it using classic numerical methods \textcolor{red}{and proposing regularization methods to overcome the difficulty accordingly occupied a large position in various kinds of studies until now.} \\
  
  \textcolor{red}{In the study of regularization methods, one may find} many studies on the homogeneous problem, i.e. $ f= 0$ in Eq. (1.1). For instance, Elden and Berntsson \cite{elden} used the logarithmic convexity method to obtain a stability result of H\"older type. Alessandrini
et al. \cite{ales} provided essentially optimal stability results, in
wide generality and under substantially minimal assumptions.   Reg\'inska and  Tautenhahn \cite{Re} presented some  stability estimates and regularization method  for a Cauchy problem for Helmholtz equation.  The homogeneous problems were also investigated by some earlier papers, such as \cite{Bo,Bo1,feng,D+N+V+2009,Qian,Re2,Tuan2}. 
\textcolor{red}{Very} recently, the inhomogeneous version of elliptic equations has been considered in \cite{Tuan}.

To the authors' knowledge, the results on regularization theory  for the Cauchy problem of an elliptic equation with nonlinear source like the problem \eqref{e1} are very rare and until now we \textcolor{red}{do} not find any results associated with a locally Lipschitz source term. 
\textcolor{red}{Additionally}, the nonlinear  case is, of course, more general and \textcolor{red}{quite close} to practical applications
than the homogeneous linear case.  \textcolor{red}{At once, one may deduce the elliptic-since Gordon equation by taking $ f(x,t,u)=\sin u$. In view of physical
phenomena models}, the motivation for the study of elliptic-sine Gordon equation comes from applications
in several areas of mathematical physics including the theory of Josephson effects, superconductors
and spin waves in ferromagnets, see e.g. \cite{Fo,Gu}. Furthermore, the Lane-Emden equation 
$ \Delta u = m u^{p},$ ($m$ is constant )
 plays a vital role in describing the structure of the polytropic stars where $p$ is called the polytropic index. Many abstract studies about this equation are the platform for the system Lane-Emden-Fowler arising in molecular biology, that received considerable mathematical attention, such as  Pohozaev-type identities, moving
plane method. For more details, the Lane-Emden equation can be referred to the book by Chandrasekhar \cite{book} and Emden et al. \cite{Lane}.

Motivated by all above  reasons, in the present paper we propose   a   new modified  method  based on a nonlinear integral equation   to regularize  Problem (1.1) with \textcolor{red}{both} globally and locally Lipschitz sources. As we know,  for a nonlinear Cauchy problem, its solution (exact solution)  can be represented in a nonlinear  integral equation (called mild solution) which contains some \textcolor{red}{unstable} terms.  The leading idea of this method is to find a suitable  integral equation approximating the exact solution.  The  work is \textcolor{red}{thus} to replace those terms by regularization terms and show that the solution of our regularized  problem
converges to the solution of the Cauchy problem as the regularization parameter tends to zero. \textcolor{red}{In the homogeneous case}, we have  many choices of stability term for regularization. However, for the nonlinear problem, the solution $u$ is \textcolor{red}{complicated} and defined by an integral equation \textcolor{red}{such that}  the \textcolor{red}{right hand side} depends on $u$.   This leads to studying nonlinear problem is very difficult, so in this paper we develop some new appropriate techniques.  For more \textcolor{red}{details}, we refer reader to next sections and the papers \cite{hughes}, \cite{Tuan}.

The paper is organized as follows.  In Section 2, the regularization method-integral equation method is introduced. In Section
3, a stability estimate is proved under a priori conditions of the exact solution. In Section 4, a generalized case of the nonlinear problem with a special type of local Lipschitz function  is remarkable. Then, a numerical example is shown in Section 5.

\section{Mathematical problem and Mild solution}
\textcolor{red}{As we introduced},  $A : D(A) \subset H \to H$ is  a positive-definite, self-adjoint operator with a compact inverse on $H$. \textcolor{red}{Therefore},  $A$ admits an orthonormal eigenbasis    $ \{ \phi_p \}_{p \geq 1}$ in $H,$ associated with the eigenvalues  such
that  \begin{eqnarray*}
0<\lambda_{1}\leq \lambda_{2}\leq \lambda_{3}\leq ... \lim _{p\rightarrow \infty} \lambda_{p} =\infty,
\end{eqnarray*}  
\textcolor{red}{and consequently we have $\left\langle Au(t),\phi_p \right\rangle=\la_p \left\langle u(t),\phi_p \right\rangle$ for all $p$.}

In practice, the data $(\varphi, g) \in H \times H $ \textcolor{red}{is obtained} by measuring at discrete nodes. Hence, instead of $(\varphi, g)$, we shall get an inexact data $(\varphi^\ep, g^\ep)\in H \times H $ satisfying
\bes
\|\varphi^\ep-\varphi\|_H  \le \ep,
\|g^\ep-g\|_H  \le \ep, \label{error}
\ens
 where the constant $\ep>0$ represents a bound on the measurement error.  \\
 
\textcolor{red}{For the sake of simplicity, we divide} the problem \eqref{e1} into three cases: homogeneous linear problem, inhomogeneous linear problem and nonlinear problem.
\subsection{Homogeneous linear problem}
Let $g, \varphi \in H $, we consider the problem of finding a function $u: [0,T] \to H$ satisfying
\bes
u_{tt}=Au, \label{e2222}
\ens
subject to conditions
\bq 
\left\{ \begin{gathered}
  u(0)=\varphi,\hfill \\
 u_t(0) =g.\hfill\\
 \end{gathered}  \right.\label{e222}
\eq
Let $u(t)= \sum\limits_{p=1}^\infty \left\langle u(t),\phi_p\right\rangle \phi_p$ be the Fourier series of $u$ in the Hilbert space $H$ . From \eqref{e2222}, we get the homogeneous second order differential equation as follows
\be
\frac{d^2}{dt^2} \left\langle u(t),\phi_p \right\rangle -\la_p \left\langle u(t),\phi_p \right\rangle = 0.
\en 
Solving this equation, \textcolor{red}{we deduce that}
\be
\left\langle u(t),\phi_p \right\rangle = A_p e^{\sqrt{\la_p}t}+ B_p e^{-\sqrt{\la_p}t}.
\en
It follows from  \eqref{e222} that $\left\langle u(0),\phi_p \right\rangle = \left\langle \varphi,\phi_p \right\rangle  $ and $ \left\langle u_t(0),\phi_p \right\rangle = \left\langle g,\phi_p \right\rangle$.  By a simple computation for coefficients $A_p$ and $B_p$,  we obtain
\begin{eqnarray} 
&& u(t)=\sum\limits_{p=1}^\infty \left[\coshh \left\langle \varphi,\phi_p \right\rangle + \frac{\sinhh}{\canla} \left\langle g,\phi_p \right\rangle\right]\phi_p.\nn\\ \label{exx}
\end{eqnarray}
From F. Browder terminology, as in \cite{henry}, \textcolor{red}{the expression \eqref{exx}} can be called the {\it mild solution} of \eqref{e2222}-\eqref{e222}.


\subsection{Inhomogeneous linear problem and nonlinear problem}
\textcolor{red}{When we consider the inhomogeneous case}, i.e. $ u_{tt}=Au +f(t)$,  its solution has the Fourier series $u(t)= \sum\limits_{p=1}^\infty  \left\langle u(t),\phi_p\right\rangle \phi_p$ where $\left\langle u(t),\phi_p\right\rangle$ satisfies \textcolor{red}{the following  equation}
\be
\frac{d^2}{dt^2} \left\langle u(t),\phi_p\right\rangle -\la_p \left\langle u(t),\phi_p\right\rangle=f(t).
\en 
Solving  this equation and using \eqref{e222}, we obtain the exact solution
\begin{eqnarray} 
&& u(t)=\sum\limits_{p=1}^\infty \left[\coshh \varphi_p + \frac{\sinhh}{\canla} g_p+\int\limits_{0}^{t}\frac{\sinhs}{\canla} f(s)ds\right]\phi_p,\nn\\ \label{ex1}
\end{eqnarray}
where \bes
g_p= \left\langle g,\phi_p\right\rangle ,~
\varphi_p=\left\langle g,\phi_p\right\rangle  ,~f_p(s)= \left\langle f(s), \phi_p\right\rangle. \label{hh3}
\ens
Recently, Tuan and his group \cite{Tuan} regularized  a simple version of the equation \eqref{ex1}  by truncation method and quasi-boundary value method. In \textcolor{red}{that} case, \textcolor{red}{$Au$ explicitly has the form $-\Delta u +k^2 u$} presenting a modified Helmholtz equation with inhomogeneous source. \textcolor{red}{Up to now}, some results on numerical regularization for nonlinear case is limited.  
For the nonlinear problem \bes
u_{tt}=Au+f(t,u(t)), \label{e22222}
\ens
subjects to conditions \eqref{e222}, we say that $ u \in C([0,T];H) $ is a mild solution of (2.8) if $u$ satisfies the integral equation 
\begin{eqnarray} 
&& u(t)=\sum\limits_{p=1}^\infty \left[\coshh \varphi_p + \frac{\sinhh}{\canla} g_p+\int\limits_{0}^{t}\frac{\sinhs}{\canla} f_{p}(u)(s)ds\right]\phi_p,\nn\\ \label{ex2}
\end{eqnarray}
where $f_{p}(u)(s) = \left\langle f(s,u(s)), \phi_p)\right\rangle$. 
The transformation  \eqref{e22222} into \eqref{ex2} is easily proved by a separation method which is similar above process. From now on, to regularize Problem (1.1), we only consider the integral equation \eqref{ex2} and find a regularization method for it \textcolor{red}{in Section 3}. The main idea of integral equation method can be found in the paper \cite{Trong} on \textcolor{red}{a class of} nonlinear backward heat equations. 
\section{Regularization and error estimate for nonlinear problem: globally Lipschitz source}
In this section, let $f: \R \times H \to H$ be a function satisfying
\begin{equation} \label{H1}
\|f(t,w) - f(t,v)\| \leq K \|w - v\|,
\end{equation}
for a positive constant $K$ independent of $w, v \in H,~t\in \R  $.\\
Since
$t>0$ , we know from \eqref{ex2}  that, when $p$ becomes large, the terms $$\coshh,\sinhh,\sinhs,$$ increases rather quickly. Thus, these terms
are the
unstable causes. Hence, to regularize the problem, we have to replace these terms by some better terms (called stability terms). In particular, $\coshh$, $\sinhh$ and $\sinhs$ are replaced by $\coshe$, $\sinhe$ and $\sinhes$, respectively. Those terms are defined as follows.
\be
\cosh^\ep(\canla t)&=&\frac{  Q(\ep,\lambda_p)e^{\canla t}   +e^{-\canla t} }{2}, \nn\\
\sinh^\ep(\canla t)&=&\frac{  Q(\ep,\lambda_p)e^{\canla t}   -e^{-\canla t} }{2},\nn\\
\sinh^\ep(\canla (t-s) )&=&\frac{  Q(\ep,\lambda_p)e^{\canla (t-s)}   -e^{-\canla (t-s)} }{2},\nn\\
Q(\ep,\lambda_p)&=& \frac{e^{-\sqrt{\lambda_p}T}}{ \ep+ e^{-\sqrt{\lambda_p}T} }.  \label{Q}\\
\en 
Then, we  approximate the exact solution $u$  in  \eqref{ex2} by the  regularized solution $v^\ep$ which defined by
\begin{align}
 v^\ep(t)&=\sum\limits_{p=1}^\infty \Big[\coshe \varphi_p +\frac{\sinhe}{\canla} g_p  +\int\limits_{0}^{t}\frac{\sinhes}{\canla}
 f_{p}(v^\ep)(s)ds\Big]\phi_p.         
           \label{vep}
\end{align}
We have the following lemma which shall be useful in next results.
\begin{lemma}
For $0 \le s \le t\le T$, we have the following inequalities 
\bes
 \coshe &\le&  \ep^{-\frac{t}{T}},\\
\frac{\sinhe}{\canla}  &\le&   \frac{ \ep^{\frac{-t}{T}} }{\sqrt{\la_1}},\\
\frac{\sinhes}{\canla} &\le& \frac{ \ep^{\frac{s-t}{T}} }{\sqrt{\la_1}}.
\ens
\end{lemma}
\begin {proof}
First, we can deduce the following inequality.
\bes
Q(\ep,\lambda_p)e^{\sqrt{\lambda_p}t} & =&
 \frac{e^{-\sqrt{\lambda_p}(T-t)}}{\ep +e^{-\sqrt{\lambda_p}T}} 
=\frac{e^{-\sqrt{\lambda_p}(T-t)}}{ \Big(\ep +e^{-\sqrt{\lambda_p}T}\Big)^{1-\frac{t}{T}} \Big(\ep +e^{-\sqrt{\lambda_p}T}\Big)^{\frac{t}{T}}  } \nn\\
&\le& \Big({\ep +e^{-\sqrt{\lambda_p}T}   }  \Big)^{\frac{-t}{T}} \le  \ep^{-\frac{t}{T}}. 
\ens
Similarly, this leads to
\bes
Q(\ep,\lambda_p)e^{\sqrt{\lambda_p}(t-s)} & =&
 \frac{e^{-\sqrt{\lambda_p}(T-t+s)}}{\ep +e^{-\sqrt{\lambda_p}T}} 
=\frac{e^{-\sqrt{\lambda_p}(T-t+s)}}{ \Big(\ep +e^{-\sqrt{\lambda_p}T}\Big)^{1-\frac{t-s}{T}} \Big(\ep +e^{-\sqrt{\lambda_p}T}\Big)^{\frac{t-s}{T}}  } \nn\\
&\le& \Big({\ep +e^{-\sqrt{\lambda_p}T}   }  \Big)^{\frac{s-t}{T}} \le  \ep^{\frac{s-t}{T}}. 
\ens
We note that $ \ep^{\frac{t}{T}} \leq 1 \leq  e^{\sqrt{\lambda_p}t} $, then \textcolor{red}{the proof is obtained}. Indeed, we see that
\be
\coshe= \frac{  Q(\ep,\lambda_p)e^{\sqrt{\lambda_p}t}   +e^{-\sqrt{\lambda_p}t}}{2} \le \ep^{-\frac{t}{T}},
\en
and 
\be
\frac{\sinhe}{\canla}= \frac{  Q(\ep,\lambda_p)e^{\sqrt{\lambda_p}t} -e^{-\sqrt{\lambda_p}t}}{2\sqrt{\lambda_p}} \le \frac{ \ep^{\frac{-t}{T}} }{\sqrt{\la_1}},
\en
and
\be
 \frac{\sinhes}{\canla}=\frac{  Q(\ep,\lambda_p)e^{\sqrt{\lambda_p}(t-s)} -e^{-\sqrt{\lambda_p}(t-s)}}{2\sqrt{\lambda_p}} \le \frac{ \ep^{\frac{s-t}{T}} }{\sqrt{\la_1}}.
\en
\end{proof}
In the next theorem, we shall study the existence and uniqueness of a (weak) regularized solution $v^\ep$ to  Problem \eqref{vep}.
\begin{theorem} 
 The  integral equation \eqref{vep}
has a unique solution $v^\ep \in  C([0; T];H)$. 
\end{theorem}
\begin{proof}
For $w \in  C([0; T];H),$ we consider the following function 
\begin{align}
F(w)(t)=\sum\limits_{p=1}^\infty \left[\coshe \varphi_p +\frac{\sinhe}{\sqrt{\la_p}} g_p+\int\limits_{0}^{t}\frac{\sinhes}{\canla}
 f_{p}(w)(s)ds\right]\phi_p. \nn
\end{align}
We claim that, for every $w, v \in C([0,T];H), m \geq 1$,
 we have
\begin{align}
\|F^m (w)(t)-F^m(v)(t)\|^2 \leq \left( \frac{K^2\ep^{-2}}{{\lambda_1}} \right)^{m}
\frac{t^m C^m}{m!}|||w -v|||^2, \label{pt1}
\end{align}
where $C = \max \{T,1\}$ and $|||.|||$ is supremum norm
in $C([0,T]; H)$.
We are going to prove this inequality by induction.
For $m =1$, using the inequality $\frac{\sinhes}{\canla} \le \frac{1}{\ep \sqrt{\la_1}}$, we obtain the following estimate
\be
\|F(w)(t)-F(v)(t)\|^2 
&=&\sum_{p=1}^\infty
\Big[\int_0^t  \frac{\sinhes}{\canla}   \left\langle f(s,w(s))-f(s,v(s)),\phi_p\right\rangle   ds\Big]^2
 \\
&\leq&  \sum_{p=1}^\infty \int_0^t
\left( \frac{\sinhes}{\canla}
\right) ^2ds \int_0^t\left|\left\langle f(s,w(s))-f(s,v(s)),\phi_p\right\rangle\right|^2 ds\\
 &\le&  \frac{t}{{\lambda_1 }\ep^2}  \int_0^t\left\|f(s,w(s))-f(s,v(s))\right\|^2 ds\nn\\
&\le&  \frac{K^2 T}{{\lambda_1} \ep^2}  \int_0^t\left\|w(s)-v(s)\right\|^2 ds \le  \frac{K^2CT}{{\lambda_1} \ep^2}   |||w-v|||^2.
\en
Thus \eqref{pt1} holds.

Next, suppose that \eqref{pt1} holds for $m =k$, then we prove that \eqref{pt1}
 holds for $m = k + 1$. We have
\begin{align*}
&\|F^{k+1}(w)(t) - F^{k+1}(v)(t)\|^2  \\
&= \sum_{p = 1}^\infty
\Big[\int_0^t \frac{\sinhes}{\canla}  \left\langle f(F^k(w))(s)
- f(F^k (v))(s),\phi_p\right\rangle   ds \Big]^2 \\
&\leq  \sum_{p=1}^\infty \int_0^t
 \left( \frac{\sinhes}{\canla}
\right) ^2ds
 \int_0^t\left|\left\langle f(F^k(w))(s)
- f(F^k (v))(s),\phi_p\right\rangle \right|^2 ds
 \\
&\le \frac{t}{{\lambda_1} \ep^2}   \int_0^t\left\|f(F^k(w))(s)
- f(F^k (v))(s)\right\|^2 ds\\
&\leq \frac{K^2 t }{{\lambda_1}\ep^2} 
 \int_0^t \left( \frac{K^2  }{{\lambda_1}\ep^2} \right)^{k}
\frac{T^k C^k}{k!}|||w -v|||^2    \\
&\leq  \left( \frac{K^2}{{\lambda_1}\ep^2} \right)^{k+1}
\frac{T^{k+1} C^{k+1}}{(k+1)!}  |||w-v|||^2.
\end{align*}
Therefore, by the induction principle, we have
\begin{equation}
\|F^m (w)(t)-F^m(v)(t)\| \leq \sqrt {\Big( \frac{K^2}{{\lambda_1}\ep^2} \Big)^{m}
\frac{T^m C^m}{m!}}|||w -v|||,
\end{equation}
for all $w, v \in C([0,T];H)$.
We consider $F: C([0,T];H)\to C([0,T];H)$.
Since
$$
\lim  _{m \to \infty} \sqrt {\Big( \frac{K^2}{{\lambda_1}\ep^2} \Big)^{m}
\frac{T^m C^m}{m!}}=0,
$$
 there exists a positive integer number $m_0$ such that
$
\sqrt {\Big( \frac{K^2}{{\lambda_1}\ep^2} \Big)^{m_0}
\frac{T^{m_0} C^{m_0}}{{m_0}!}} <1,
$
and $F^{m_0}$ is a contraction. It follows that the equation
$F^{m_0} (w) = w$ has a unique solution
$v^\epsilon \in C([0,T]; H)$.
We claim that $F(v^\epsilon) = v^\epsilon$. In fact, we have $F^{m_0} (F(v^\epsilon)) = F(v^\epsilon)$ because of $F(F^{m_0} (v^\epsilon)) = F(v^\epsilon)$. Then, the uniqueness
of the fixed point of $F^{m_0}$ leads to  $F(v^\epsilon) = v^\epsilon$;
i.e., the equation $F(w) = w$ has a unique solution
$v^\epsilon \in C([0,T];H)$. Hence, the result is proved completely.
\end{proof}
In the next Theorem, we make a study on the stability of solution of the considered problem.
\begin{theorem}
{ Let  $(\varphi, g) \in  H \times H$. Then the function $u^\ep \in C([0,T]; H)$ as in (3.11) depends continuously on $(\varphi, g)$ for any  $\ep>0$. 
Let $u^{1,\ep}$ and $u^{2,\ep}$ be two solutions of \eqref{vep} corresponding to the values $(\varphi^1, g^1)$ and $(\varphi^2, g^2)$ respectively, then
\be
\|u^{1,\ep}(t)-u^{2,\ep}(t)\|^2  \le   3 \exp\left\{ \frac{3TK^2t}  {{\la_1}} \right\} \ep^{\frac{-2t}{T}} \left(  \|\varphi^1-\varphi^2\|^2 +\frac{\|g^1-g^2\|^2  }{\la_1  }   \right).
\en
}
\end{theorem}
\begin{proof}
For $i=1,2 $, we have
\begin{align}
 u^{i,\ep}(t)&=\sum\limits_{p=1}^\infty \left[\coshe \varphi^i_p +\frac{\sinhe}{\canla} g^i_p
 +\int\limits_{0}^{t}\frac{\sinhes}{\canla}
 f_{p}(u^{i,\ep})(s)ds\right]\phi_p .     \label{ui}          
\end{align}
It follows from \eqref{ui} that
\be
\|u^{1,\ep}(t)-u^{2,\ep}(t)\|^2&=&\sum\limits_{p=1}^\infty \Big[ \coshe (\varphi^1_p-\varphi^2_p) +\frac{\sinhe}{\canla}(g^1_p-g^2_p)\nn\\
&+&\int\limits_{0}^{t}\frac{\sinhes}{\canla}
 \Big(f_{p}(u^{1,\ep})(s)- f_{p}(u^{2,\ep})(s)  \Big)  ds\Big]^2\nn\\
&\le&  \sum\limits_{p=1}^\infty \Big[3 |\coshe|^2(\varphi^1_p-\varphi^2_p)^2 +3\Big(   \frac{\sinhe}{\canla}   \Big)^2 (g^1_p-g^2_p)^2\nn\\
&\quad +&3t\int\limits_{0}^{t}\left( \frac{\sinhes}{\canla}
\right)^2 
 \Big(f_{p}(u^{1,\ep})(s)- f_{p}(u^{2,\ep})(s)  \Big)^2  ds\Big].
\en
Using the Lipschitz property of $f$, we get the following inequality.
\be
\|u^{1,\ep}(t)-u^{2,\ep}(t)\|^2 &\le& 3 \ep^{-\frac{2t}{T}} \|\varphi^1-\varphi^2\|^2+ \frac{3}{{\la_1}} \ep^{-\frac{2t}{T}} \|g^1-g^2\|^2\nn\\
&~~~~~+&\frac{3t}  {{\la_1}} \int\limits_{0}^{t} \ep^{\frac{2s-2t}{T}} \| f(s,u^{1,\ep}(s))- f(s,u^{2,\ep}(s))    \|^2ds\nn\\
&\le&  3 \ep^{-\frac{2t}{T}}  \left(  \|\varphi^1-\varphi^2\|^2 +\frac{\|g^1-g^2\|^2  }{\la_1  }   \right)+\frac{3tK^2}  {{\la_1}} \int\limits_{0}^{t} \ep^{\frac{2s-2t}{T}} \| u^{1,\ep}(s)-u^{2,\ep}(s)\|^2ds.
\en
This implies that
\begin{align*}
\ep^{\frac{2t}{T}}\|u^{1,\ep}(t)-u^{2,\ep}(t)\|^2 
\le 3  \left(  \|\varphi^1-\varphi^2\|^2 +\frac{\|g^1-g^2\|^2  }{\la_1  }   \right)
+\frac{3tK^2}  {{\la_1}} \int\limits_{0}^{t} \ep^{\frac{2s}{T}} \| u^{1,\ep}(.,s)-u^{2,\ep}(.,s)    \|^2ds.
\end{align*}
Applying Gronwall's inequality, we have
\begin{align}
\ep^{\frac{2t}{T}}\|u^{1,\ep}(t)-u^{2,\ep}(t)\|^2  \le   3 \exp\left\{ \frac{3T^2K^2}  {{\la_1}} \right\}  \left(  \|\varphi^1-\varphi^2\|^2 +\frac{\|g^1-g^2\|^2  }{\la_1  }   \right).
\end{align}
Thus, it turns out that
\be
\|u^{1,\ep}(t)-u^{2,\ep}(t)\|^2  \le   3 \exp\left\{ \frac{3T^2K^2}  {{\la_1}} \right\} \ep^{\frac{-2t}{T}} \left(  \|\varphi^1-\varphi^2\|^2 +\frac{\|g^1-g^2\|^2  }{\la_1  }   \right).
\en
\end{proof}
\begin{theorem}
Suppose that Problem (1.1) has a weak solution $u  $ which satisfies 
\bes
\sum\limits_{p=1}^\infty e^{2\sqrt{\lambda_p}(T-t)}\Big( \left\langle u(t),\phi_p\right\rangle + \frac{\left\langle u_t(t),\phi_p\right\rangle}{\sqrt{\la_p}}   \Big)^2 \le P^2. \label{dk1}
\ens
for a positive number 
$P$, and let $(\varphi^\ep, g^\ep)  \in H\times H $ be measured data \textcolor{red}{satisfying} \eqref{error}. Then, we can construct a regularized solution $U^\ep$  such that

\bq 
 \left\{ \begin{gathered}
\|  U^\ep (t)-u(t)\| \le C \ep^{1-\frac{t}{T}}, ~~~t \in [0,T) \hfill \\
 \|  U^\ep (t)-u(t)\|  \le (D+C) \sqrt{\frac{T}{\ln (\frac{1}{\ep})}  } , ~~~t=T,  \hfill \\
 \end{gathered}  \right.\label{e11111}
\eq
where 
\be
C&=&2  P \exp\left\{\frac{ K^2 T^2}{\la_1}\right\} +\sqrt{3(1+\frac{1}{\la_1})} \exp\left\{ \frac{3TK^2}  {{2\la_1}}\right\},\\
D&= & \sup_{0\le t\le T}\|u_t(t)\|.
\en
\end{theorem}
\begin{proof}
Differentiating \eqref{ex2} with respect to $t$ gives
\begin{align} 
& \Big\langle u_t(t), \phi_p \Big\rangle =\quad \sqrt{\la_p}\left[\sinhh \varphi_p +\frac{\coshh}{\canla} g_p+\int\limits_{0}^{t}\frac{\coshs}{\canla} f_{p}(u)(s)ds\right].\nn\\ \label{k11}
\end{align}
Dividing \eqref{k11} by $\sqrt{\la_p} $ and adding the result obtained to \eqref{ex2}, we get
\begin{align}
\Big\langle u(t),\phi_p\Big\rangle + \frac{\Big\langle u_t(t),\phi_p\Big\rangle}{\sqrt{\la_p}}={e^{\sqrt{\lambda_p}t}} \varphi_p +   \frac{e^{\sqrt{\lambda_p}t}}{\sqrt{\lambda_p}}      g_p+\int\limits_{0}^{t} \frac{e^{\sqrt{\lambda_p}(t-s)}} {\sqrt{\lambda_p}}
 f_{p}(u)(s)ds. \label{k12}
\end{align}
Combining  \eqref{ex2},\eqref{vep} and \eqref{k12} yields
\begin{align}
&\Big\langle v^\ep(t)-u(t),\phi_p\Big\rangle =\nn\\
&=\Big(Q(\ep,\lambda_p)-1\Big) \Big( {e^{\sqrt{\lambda_p}t}} \varphi_p +   \frac{e^{\sqrt{\lambda_p}t}}{\sqrt{\lambda_p}}      g_p +\int\limits_{0}^{t} \frac{e^{\sqrt{\lambda_p}(t-s)}} {\sqrt{\lambda_p}}
 f_{p}(u)(s)ds\Big)\nn\\
&\quad +\int\limits_{0}^{t}\frac{\sinhes}{\canla}
 f_{p}(v^\ep)(s)ds-\int\limits_{0}^{t}\frac{\sinhes}{\canla} f_{p}(u)(s)ds\nn\\
&=\Big(Q(\ep,\lambda_p)-1\Big) \Big( \left\langle u(t),\phi_p\right\rangle+ \frac{\left\langle u_t(t),\phi_p\right\rangle}{\sqrt{\la_p}}   \Big)\nn\\
&\quad + \int\limits_{0}^{t}\frac{\sinhes}{\canla} \Big(
 f_{p}(v^\ep)(s) -f_p(u)(s)\Big) ds.  \label{qt}
\end{align}
Therefore, the following inequality is obtained directly.
\begin{align}
&\Big| \left\langle v^\ep(t)-u(t), \phi_p \right\rangle \Big|^2 \nn\\
&\quad \le 2\Big(Q(\ep,\lambda_p)-1\Big)^2 \Big( \left\langle u(t),\phi_p\right\rangle + \frac{\left\langle u_t(t),\phi_p\right\rangle}{\sqrt{\la_p}}   \Big)^2 \nn\\
&\quad + 2 t^2 \int\limits_{0}^{t}\left( \frac{\sinhes}{\canla}
\right)^2 \Big(
 f_{p}(v^\ep)(s) -f_p(u)(s)\Big)^2 ds\nn\\
&\quad \le  2\ep^2  \left(\frac{e^{-\sqrt{\lambda_p}(T-t)}}{\ep +e^{-\sqrt{\lambda_p}T}} \right)^2 e^{2\sqrt{\lambda_p}(T-t)}\Big( \left\langle u(t),\phi_p\right\rangle + \frac{\left\langle u_t(t),\phi_p\right\rangle}{\sqrt{\la_p}}   \Big)^2 \nn\\
&\quad + 2 t^2 \int\limits_{0}^{t}\left( \frac{\sinhes}{\canla}
\right)^2 \Big|
 f_{p}(v^\ep)(s) -f_p(u)(s)\Big|^2 ds.\nn
\end{align}
At this time, applying Lemma 3.1 leads to
\be
\Big| \left\langle v^\ep(t)-u(t),\phi_p \right\rangle \Big|^2 &\le& 2 \ep^{\frac{2T-2t}{T}} e^{2\sqrt{\lambda_p}(T-t)}\Big( \left\langle u(t),\phi_p\right\rangle + \frac{\left\langle u_t(t),\phi_p\right\rangle}{\sqrt{\la_p}}   \Big)^2\nn\\
&~~+&2 t^2 \int\limits_{0}^{t}  \frac{ \ep^{\frac{2s-2t}{T}} }{{\la_1}}     \Big(
 f_{p}(v^\ep)(s) -f_p(u)(s)\Big)^2 ds.
\en
Using  Lipschitz property of $f$ and the priori assumption \eqref{dk1}, we get
\be
\|v^\ep(t)-u(t)\|^2 &=& \sum\limits_{p=1}^\infty  \Big| \left\langle v^\ep(t)-u(t),\phi_p\right\rangle \Big|^2\nn\\
&\le& 2 \ep^{\frac{2T-2t}{T}} \sum\limits_{p=1}^\infty e^{2\sqrt{\lambda_p}(T-t)}\left( \left\langle u(t),\phi_p\right\rangle+ \frac{\left\langle u_t(t),\phi_p\right\rangle}{\sqrt{\la_p}}   \right)^2\nn\\
&+& 2 t^2 \int\limits_{0}^{t}  \frac{ \ep^{\frac{2s-2t}{T}} }{{\la_1}} \sum\limits_{p=1}^\infty    \Big|
 f_{p}(v^\ep)(s) -f_p(u)(s)\Big|^2 ds\nn\\
&\le& 2 \ep^{\frac{2T-2t}{T}} P^2 + \frac{2 T^2}{\la_1} \int\limits_{0}^{t}  { \ep^{\frac{2s-2t}{T}} }  \| f(v^\ep)(s) -f(u)(s)\|^2 ds
\nn\\
&\le& 2 \ep^{\frac{2T-2t}{T}} P^2 + \frac{2 K^2 T^2}{\la_1} \int\limits_{0}^{t}  { \ep^{\frac{2s-2t}{T}} }  \| v^\ep(s) -u(s)\|^2 ds.
\en
Multiplying $e^{\frac{2t}{T}}$ by both sides, we thus have
\be
\ep^{\frac{2t}{T}} \|v^\ep(.,t)-u(.,t)\|^2 \le 2 \ep^2 P^2 + \frac{2 K^2 T^2}{\la_1} \int\limits_{0}^{t}  { \ep^{\frac{2s}{T}} }  \| v^\ep(s) -u(s)\|^2 ds.
\en
Applying Gronwall's inequality, we deduce that
\be
\ep^{\frac{2t}{T}} \|v^\ep(t)-u(t)\|^2 \le  2  P^2 \exp\left\{\frac{2 K^2 T^2t}{\la_1}\right\} \ep^2.
\en
By simplification, we conclude that
\bes
\|v^\ep(t)-u(t)\| \le 2  P \exp\left\{\frac{ K^2 T^2t}{\la_1}\right\} \ep^{1-\frac{t}{T}}. \label{p1}
\ens
On the other hand, notice that the following integral equation
\bes
 u^\ep(t)&=\sum\limits_{p=1}^\infty \Big[\coshe \varphi_p^\ep +\frac{\sinhe}{\canla} g_p^\ep+\int\limits_{0}^{t}\frac{\sinhes}{\canla}
 f_{p}(u^\ep)(s)ds\Big]\phi_p.          
           \label{uep}
\ens
has a unique solution $u^\ep \in C([0,T];H)$.
Then, by using Lemma 3.1, we obtain
\bes
\|u^\ep(t)-v^\ep(t)\| &\le&  \sqrt{3 \exp\left\{ \frac{3TK^2t}  {{\la_1}} \right\} \ep^{\frac{-2t}{T}} \left(  \|\varphi^\ep-\varphi\|^2 +\frac{\|g^\ep-g\|^2  }{\la_1  }   \right)}\nn\\
&\le& \sqrt{3(1+\frac{1}{\la_1})} \exp\left\{ \frac{3TK^2t}  {{2\la_1}}\right\}\ep^{1-\frac{t}{T}}. \label{p2}
\ens
Combining \eqref{p1} and \eqref{p2} and using the triangle inequality, we have
\be
\|u^\ep(t)-u(t)\|  &\le& \|u^\ep(t)-v^\ep(t)\|+\|v^\ep(t)-u(t)\|\nn\\
&\le& 2  P \exp\left\{\frac{ K^2 T^2t}{\la_1}\right\} \ep^{1-\frac{t}{T}}+\sqrt{3(1+\frac{1}{\la_1})} \exp\left\{ \frac{3TK^2t}  {{2\la_1}}\right\}\ep^{1-\frac{t}{T}}\nn\\
&\le& C \ep^{1-\frac{t}{T}}.
\en
Moreover, we get
\be
\|u(T)-u^\epsilon(t)\| &\le&  \|u(T)-u(t)\|+\|u(t)-u^\epsilon(t)\|  \\
&\le& \sup_{0\le t\le T}\|u_t(t)\|({T-t})+C \ep^{1-\frac{t}{T}}.
\en
For every $\ep>0$, there exists a unique  $t_\ep \in (0,T) $ such that 
\bes
{(T-t_\ep)^2}=\ep^{2-\frac{2t_\ep}{T}}. \label{tep}
\ens
It implies that $
\frac{\ln(T-t_\ep)}{T-t_\ep}=\frac{\ln \ep}{T}.
$
Using the inequality $\ln t > -\frac{1}{t}$ for every $t>0$, we obtain
$$
T-t_\ep<  \sqrt{\frac{T}{\ln (\frac{1}{\ep})}  }.
$$
This leads to
\bes
\|u(T)-u^\epsilon(t_\ep)\|   \le  (D+C) \sqrt{\frac{T}{\ln (\frac{1}{\ep})}  }, \label{dg1}
\ens
where $D=  \sup_{0\le t\le T}\|u_t(t)\|. $
Let $U^\ep$ be defined as follows
\bq 
U^\ep(t)= \left\{ \begin{gathered}
  u^\ep (t), ~~~t \in [0,T), \hfill \\
 u^\ep (t_\ep), ~~~t=T.  \hfill \\
 \end{gathered}  \right.
\eq
\textcolor{red}{Therefore, together with \eqref{dg1} we obtain}
\be
\|U^\ep(t)-u(t)\| = \|u^\ep (t)-u(t)\| \le C \ep^{1-\frac{t}{T}}.
\en
and 
\be
\|U^\ep(T)-u(T)\| = \|u^\ep (t_\ep)-u(t)\| \le (D+C) \sqrt{\frac{T}{\ln (\frac{1}{\ep})}  }.
\en
\end{proof}
which \textcolor{red}{are the desired estimates}.
\begin{remark}
 The condition in \eqref{dk1} is accepted and natural. 
If the source function $f=0$ then from (3.24), we have
\begin{align}
\left\langle u(t),\phi_p\right\rangle+ \frac{\left\langle u_t(t),\phi_p\right\rangle}{\sqrt{\la_p}}={e^{\sqrt{\lambda_p}t}} \varphi_p +   \frac{e^{\sqrt{\lambda_p}t}}{\sqrt{\lambda_p}}      g_p. \label{k122}
\end{align}
By letting $t=T$, we have
\begin{align}
\left\langle u(T),\phi_p\right\rangle + \frac{\left\langle u_t(T),\phi_p\right\rangle}{\sqrt{\la_p}}={e^{\sqrt{\lambda_p}T}} \varphi_p +   \frac{e^{\sqrt{\lambda_p}T}}{\sqrt{\lambda_p}}      g_p. \label{k123}
\end{align}
Combining \eqref{k122} and \eqref{k123}, we obtain
\begin{align}
{e^{\sqrt{\lambda_p}(T-t)}}\left( \left\langle u(t),\phi_p\right\rangle + \frac{\left\langle u_t(t),\phi_p\right\rangle}{\sqrt{\la_p}}\right)=\left\langle u(T),\phi_p\right\rangle+ \frac{\left\langle u_t(T),\phi_p\right\rangle}{\sqrt{\la_p}}. \label{k1233}
\end{align}
Then, it follows that
\be
\sum\limits_{p=1}^\infty e^{2\sqrt{\lambda_p}(T-t)}\Big( \left\langle u(t),\phi_p\right\rangle+ \frac{\left\langle u_t(t),\phi_p\right\rangle}{\sqrt{\la_p}}   \Big)^2&=&\sum\limits_{p=1}^\infty  \left( \left\langle u(T),\phi_p\right\rangle+ \frac{\left\langle u_t(T),\phi_p\right\rangle}{\sqrt{\la_p}} \right)^2.
\en
\end{remark}

\begin{remark}
\textcolor{red}{Our construction can be successfully applied to the modified Helmholtz equation where $A=-\frac{\partial}{\partial x^2} + k^2 Id$ and $f(x,t,u)=f(x,t)$ in \eqref{e1} over higher Sobolev spaces $H^s,s\in\mathbb{N}$ in a finite bounded and connected domain. In fact, one may observe the one-dimensional modified Helmholtz equation in the form of
\bq 
\left\{ \begin{gathered}
 u_{tt}+u_{xx}-k^2 u(x,t)=f(x,t), (x,t) \in   (0,\pi)\times(0,1), \hfill \\
 u(0,t)=u(1,t)=0, t\in (0,1),   \hfill \\
 u(x,0)=\varphi(x),\; u_{t}(x,0)=g(x), x\in(0,\pi).   \hfill\\
 \end{gathered}  \right.\label{mhe}
\eq
We thus obtain an error estimate between the exact solution and regularized solution in the space $H^{s}_{0}$ but before claiming the estimate let us consider the following lemma.}
\end{remark}

\begin{lemma}
{ Let $s>0, X\ge 0$. Then for all $0 \le t \le T$ and $0<\ep<1$, we have
\be
\frac{\ep }{(1+X)^s(\ep+e^{-TX})} \le  C(s)\left(\frac{T}{\ln(1/\ep)}\right)^s. 
\en
where $C(s)=s^s e^{1-s} (1+T^{-s})$. 
}
\end{lemma}
\begin{proof}
{\it Case 1.}  $X \in [0, \frac{1}{T}]$.  It is clear to see that
\be
\frac{\ep }{(1+X)^s(\ep+e^{-TX})}\le \frac{\ep}{(1+X)^s e^{-TX}} \le \ep e^{TX}
\le e\ep.\en
From the inequality $\ep \le (\frac{s}{e})^s  \left(\frac{1}{\ln(1/\ep)}\right)^s$, we get
\bes
\frac{\ep }{(1+X)^s(\ep+e^{-TX})} \le s^s e^{1-s} \left(\frac{1}{\ln(1/\ep)}\right)^s \le s^s e^{1-s} (1+T^{-s})\left(\frac{T}{\ln(1/\ep)}\right)^s . \label{mm1}
\ens
{\it  Case 2.} $X >\frac{1}{T}$. Set $e^{-TX}=\ep Y$. Then, we obtain
\bes
\frac{\ep }{(1+X)^s(\ep+e^{-TX})}&=&\frac{\ep }{\ep+\ep Y} \left(\frac{T}{T-\ln(\ep Y)}\right)^s\nn\\
&=&\frac{1}{1+Y} \left(\frac{T}{T-\ln(\ep Y)}\right)^s\nn\\
&=&\frac{1}{1+Y} \left(\frac{T}{\ln(1/\ep)}\right)^s\left(\frac{-\ln(\ep)}{T-\ln(\ep Y)}\right)^s\nn\\
&=&\left(\frac{T}{\ln(1/\ep)}\right)^s\frac{1}{1+Y} \left(\frac{-\ln(\ep)}{T-\ln(\ep Y)}\right)^s. \label{nn3}
\ens
We continue to estimate the term $\frac{1}{1+Y} \left(\frac{-\ln(\ep)}{T-\ln(\ep Y)}\right)^s$.\\
If $0<Y \le 1$ then $0<-\ln (\ep) <-\ln(\ep Y)$, thus  
\bes
\frac{1}{1+Y} \left(\frac{-\ln(\ep)}{T-\ln(\ep Y)}\right)^s<1, \label{nn1}
\ens
else if $Y > 1$ then $\ln Y>0$ and $\ln (\ep Y)=-TX<-1$ due to the assumption $X \in (\frac{1}{T},\infty)$.  Therefore $\ln Y (1+\ln (\ep Y) )\le 0$. This implies that
\be
0<\frac{-\ln \ep}{T-\ln (\ep Y)}<\frac{-\ln \ep}{-\ln (\ep Y)}<1+\ln Y.
\en
Hence, in this case, we get
\be
\frac{1}{1+Y} \left(\frac{-\ln(\ep)}{T-\ln(\ep Y)}\right)^s<\frac{\left(1+\ln Y\right)^s}{Y}=\left(1+\ln Y\right)^s {Y^{-1}}.
\en
We set $g(Y)=\left(1+\ln Y\right)^s {Y^{-1}}$ for $Y>e^{-1}$. Then, taking the derivative of this function is to get
\be
g'(Y)=\left(1+\ln Y\right)^{s-1} {Y^{-2}}\left(s-1-\ln Y\right)
\en
The function $g$ has maximum at the point $Y_0$ such that $g'(Y_0)=0$. This implies that
$Y_0=e^{s-1}$. Therefore, it leads to the following inequality.
\bes
\sup_{Y\ge 1} \left(1+\ln Y\right)^s {Y^{-1}} \le g(Y_0)=s^s e^{1-s}. \label{nn2}
\ens
Combining \eqref{nn1} and \eqref{nn2}, we have 
\be
\frac{1}{1+Y} \left(\frac{-\ln(\ep)}{T-\ln(\ep Y)}\right)^s \le s^s e^{1-s} .
\en
From \eqref{nn3}, we will see that
\bes
\frac{\ep }{(1+X)^s(\ep+e^{-TX})} &\le& s^s e^{1-s} \left(\frac{T}{\ln(1/\ep)}\right)^s \le C(s)\left(\frac{T}{\ln(1/\ep)}\right)^s  .\label{mm2}
\ens
\textcolor{red}{which implies the proof of lemma.}
\end{proof}

Now we shall use this lemma to evaluate the error estimate said in the above remark. Since \eqref{qt}, we obtain
\be
\Big| \left\langle v^\ep(t)-u(t),\phi_p\right\rangle \Big|^2&=&
\Big(Q(\ep,\lambda_p)-1\Big)^2 \Big(\left\langle u(t),\phi_p\right\rangle+ \frac{\left\langle u_t(t),\phi_p\right\rangle}{\sqrt{\la_p}}   \Big)^2\nn\\
&=&\frac{\ep^2}{ (1+\sqrt{\la_p} )^{2s} \Big(\ep+ e^{-\sqrt{\lambda_p}T})^2 }(1+\sqrt{\la_p} )^{2s}\Big(\left\langle u(t),\phi_p\right\rangle+ \frac{\left\langle u_t(t),\phi_p\right\rangle}{\sqrt{\la_p}}   \Big)^2.
\en 

\textcolor{red}{
Let us note that it is straightforward to check that the eigenvalues in this case are $\lambda_{p} = p^2$ then for any $u\in H^s$,
\[\sum_{p=1}^{\infty}(1+p)^{2s}|\left\langle u(t),\phi_p \right\rangle|^2 \le 2\|u\|^2_{H^s}.\]
}

Hence, we conclude that
\be
\Big\|v^\ep(t)-u(t)\Big\|^2&=& \sum_{p=1}^\infty \Big|\left\langle v^\ep(t)-u(t),\phi_p\right\rangle \Big|^2 \nn\\
&\le& C^2(s) \left(\frac{T}{\ln(1/\ep)}\right)^{2s} \sum_{p=1}^\infty (1+\sqrt{\la_p} )^{2s}\Big(\left\langle u(t),\phi_p\right\rangle+ \frac{\left\langle u_t(t),\phi_p\right\rangle}{\sqrt{\la_p}}   \Big)^2\nn\\
&\le& C^2(s) D \left(\frac{T}{\ln(1/\ep)}\right)^{2s} \left(\|u(t)\|^2_{H^s}+\|u'(t)\|^2_{H^s}\right).
\en
 
\section{Regularization and error estimate for nonlinear problem: locally Lipschitz source}

Section 3 only \textcolor{red}{regularizes} problems in which $f$ is a global Lipschitz function, \textcolor{red}{ so it still restricts the applicability of the method in} a small field of study. We can list some functions such as $f (x) = \sin x, \arctan x, \frac{1}
{x^2+1}$, then observe that the class
of space functions is very small. From the point of view, we tend to establish the error estimate for a bigger class $a(t, u)f (t, u)$ where $a(t, u)$
will be defined later.\\
We still pay more attention to the problem \eqref{e1}. Until now, we \textcolor{red}{do} not find any results associated with a local Lipschitz function in the right hand side of \eqref{e1}. \textcolor{red}{Therefore}, in this section we are going to introduce the main idea to solve  the problem \eqref{e1} with a special generalized case of source term,
\bq 
\left\{ \begin{gathered}
 u_{tt}= A u +a(t,u(t)) f(t,u(t)), t \in  (0,T), \hfill \\
  u(0)=\varphi,  \hfill \\
 u_t(0)=g    \hfill\\
 \end{gathered}  \right.\label{e11}
\eq
where $a: [0,T] \times H \to H$ satisfies that $  \| a(x,t,u)  \|    \le M  $  and   the Lipschitz condition
 \be
\| a(t,u) - a(t,v)     \|  &\le&  N    \| u-v \|, 
\en
\textcolor{red}{and} the function  $f$ satisfies the condition \eqref{H1}.\\
Let $G(t,u(t))=a(t,u(t)) f(t,u(t))$, a mild solution of \eqref{e11} satisfies the following integral equation
\begin{eqnarray} 
&& u(t)=\sum\limits_{p=1}^\infty \left[\coshh \varphi_p + \frac{\sinhh}{\canla} g_p+\int\limits_{0}^{t}\frac{\sinhs}{\canla} G_{p}(u)(s)ds\right]\phi_p\nn\\ \label{exx}
\end{eqnarray}
Our result is in the next theorem.

\begin{theorem}
Let $u$ be defined by \eqref{exx}. Then the problem
\begin{align}
 v^\ep(t)&=\sum\limits_{p=1}^\infty \Big[\coshe \varphi_p^\ep +\frac{\sinhe}{\canla} g_p^\ep  +\int\limits_{0}^{t}\frac{\sinhes}{\canla}
 G_{p}(v^\ep)(s)ds\Big]\phi_p.         
           \label{uep}
\end{align}
has a unique solution $v^\ep$ and the following estimate is obtained:
\be
\|u^\ep-u\| \le C \ep^{1-\frac{t}{T}}
\en
for C is a constant independent of $\ep$. 
\end{theorem}

\begin{proof}

We divide the proof into three steps.\\
{\bf Step 1.} The existence of $v^\ep$ is easily proved by fixed point theory as in Section 3 and we omit it here.  \\
{\bf Step 2.} Error estimate $\|v^\ep-u\|$ \textcolor{red}{is obtained in this step} where $v^\ep$ satisfies\\
\begin{align}
 v^\ep(t)&=\sum\limits_{p=1}^\infty \Big[\coshe \varphi_p +\frac{\sinhe}{\canla} g_p  +\int\limits_{0}^{t}\frac{\sinhes}{\canla}
 G_{p}(v^\ep)(s)ds\Big]\phi_p.         
           \label{vepp}
\end{align}

Due to the orthonormal basis $\phi_p$ and the explicit formula of $v^{\epsilon}$ and $u$, we obtain
\begin{align}
&\left\langle v^\ep(t)-u(t),\phi_p\right\rangle=\nn\\
&=\Big(Q(\ep,\lambda_p)-1\Big) \Big( {e^{\sqrt{\lambda_p}t}} \varphi_p +   \frac{e^{\sqrt{\lambda_p}t}}{\sqrt{\lambda_p}}      g_p +\int\limits_{0}^{t} \frac{e^{\sqrt{\lambda_p}(t-s)}} {\sqrt{\lambda_p}}
 G_{p}(u)(s)ds\Big)\nn\\
&\quad +\int\limits_{0}^{t}\frac{\sinhes}{\canla}
 G_{p}(v^\ep)(s)ds-\int\limits_{0}^{t}\frac{\sinhes}{\canla} G_{p}(u)(s)ds\nn\\
&=\Big(Q(\ep,\lambda_p)-1\Big) \Big( \left\langle u(t),\phi_p\right\rangle+ \frac{\left\langle u_t(t),\phi_p\right\rangle}{\sqrt{\la_p}}   \Big)\nn\\
&\quad + \int\limits_{0}^{t}\frac{\sinhes}{\canla} \Big(
 G_{p}(v^\ep)(s) -G_p(u)(s)\Big) ds.
\end{align}
Using  Lipschitz property of $f$ and the priori assumption \eqref{dk1}, we get
\be
\|v^\ep(t)-u(t)\|^2 &=& \sum\limits_{p=1}^\infty  \Big| \left\langle v^\ep(t)-u(t),\phi_p\right\rangle \Big|^2\nn\\
&\le& 2 \ep^{\frac{2T-2t}{T}} \sum\limits_{p=1}^\infty e^{2\sqrt{\lambda_p}(T-t)}\left( \left\langle u(t),\phi_p\right\rangle + \frac{\left\langle u_t(t),\phi_p\right\rangle }{\sqrt{\la_p}}   \right)^2\nn\\
&+& 2 t^2 \int\limits_{0}^{t}  \frac{ \ep^{\frac{2s-2t}{T}} }{{\la_1}} \sum\limits_{p=1}^\infty    \Big|
 G_{p}(v^\ep)(s) -G_p(u)(s)\Big|^2 ds\nn\\
&\le& 2 \ep^{\frac{2T-2t}{T}} P^2 + \frac{2 T^2}{\la_1} \int\limits_{0}^{t}  { \ep^{\frac{2s-2t}{T}} }  \| G(v^\ep)(s) -G(u)(s)\|^2 ds.
\en
On the other hand, we have
\be
 \| G(v^\ep)(t) -G(u)(t)\| &=&  \Big\| a(t,v^\ep(t)) f(t,v^\ep(t)) -a(t,u(t)) f(t,u(t)) \Big\| \nn\\
&\le&  \Big\| a(t,v^\ep(t)) \Big\|     \Big\| f(t,v^\ep(t)) -f(t,u(t)) \Big\| +\nn\\
&&~~~~~+ \Big\| f(t,u(t)) \Big\|     \Big\| a(t,v^\ep(t)) -a(t,u(t)) \Big\|   \nn\\
&\le&  M K  \Big\| v^\ep(t) - u(t) \Big\| + N \Big\| f(t,u(t)) \Big\|   \Big\| v^\ep(t) - u(t) \Big\|. 
\en
\textcolor{red}{It follows from (3.10) that}
\be
\| f(t,u(t)) \Big\|  \le \| f(t,0) \Big\| + K  \|u(t)\| \le Q+ K  \|u(t)\|
\en
Moreover, we see that
\be
\|u(t)\|&=& \sqrt{   \sum\limits_{p=1}^\infty | \left\langle u(t), \phi_p\right\rangle | ^2  } \nn\\
&\le&  \sqrt{\sum\limits_{p=1}^\infty e^{2\sqrt{\lambda_p}(T-t)}\Big( \left\langle u(t),\phi_p\right\rangle + \frac{\left\langle u_t(t),\phi_p\right\rangle }{\sqrt{\la_p}}   \Big)^2} \le P.
\en
Therefore, we deduce
\be
\| G(v^\ep)(t) -G(u)(t)\| \le \Big( MK+ NQ+NKP \Big) \Big\| v^\ep(t) - u(t) \Big\|,
\en
\textcolor{red}{which leads to}
\be
\|v^\ep(t)-u(t)\|^2 \le 2 \ep^{\frac{2T-2t}{T}} P^2 + \frac{2 T^2}{\la_1} \Big( MK+ NQ+NKP \Big) \int\limits_{0}^{t}  { \ep^{\frac{2s-2t}{T}} }  \| v^\ep(s) -u (s)\|^2 ds.
\en
Thus, we have the following inequality.
\be
\ep^{\frac{2t}{T}} \|v^\ep(t)-u(t)\|^2 \le 2 \ep^{2} P^2 + \frac{2 T^2}{\la_1} \Big( MK+ NQ+NKP \Big) \int\limits_{0}^{t}  { \ep^{\frac{2s}{T}} }  \| v^\ep(s) -u (s)\|^2 ds.
\en
Applying Gronwall's inequality, we obtain
\be
\ep^{\frac{2t}{T}} \|v^\ep(t)-u(t)\|^2 \le 2 P^2 \exp\Big\{ \frac{2 T^2t}{\la_1} \Big( MK+ NQ+NKP \Big)   \Big\}   \ep^{2}   .
\en
Hence, we finish the step.
\bes
 \|v^\ep(t)-u(t)\| \le 2 P \exp\Big\{ \frac{ T^2t}{\la_1} \Big( MK+ NQ+NKP \Big)   \Big\}   \ep^{1-\frac{t}{T}}   . \label{ss1}
\ens

{\bf Step 3.} Error estimate $\|v^\ep-u^\ep\|$ \textcolor{red}{is obtained in this step by a similar way. Indeed,} we have
\be
 \| G(v^\ep)(t) -G(u^\ep)(t)\| &=&  \Big\| a(t,v^\ep(t)) f(t,v^\ep(t)) -a(t,u^\ep(t)) f(t,u^\ep(t)) \Big\| \nn\\
&\le&  \Big\| a(t,v^\ep(t)) \Big\|     \Big\| f(t,v^\ep(t)) -f(t,u^\ep(t)) \Big\| +\nn\\
&&~~~~~+ \Big\| f(t,v^\ep(t)) \Big\|     \Big\| a(t,v^\ep(t)) -a(t,u^\ep(t)) \Big\|   \nn\\
&\le&  M K  \Big\| v^\ep(t) - u^\ep(t) \Big\| + N \Big\| f(t,v^\ep(t)) \Big\|   \Big\| v^\ep(t) - u^\ep(t) \Big\|. 
\en
On the other hand, it is similar that from (3.10) that
\be
\| f(t,v^\ep(t)) \Big\|  &\le& \| f(t,0) \Big\| + K  \|v^\ep(t)\| \le Q+ K  \|v^\ep(t)\| \nn\\
&\le& Q+ K  \left(  \|u(t)\|+ 2 P \exp\Big\{ \frac{ T^2t}{\la_1} \Big( MK+ NQ+NKP \Big)   \Big\}   \ep^{1-\frac{t}{T}}   \right)\nn\\
&\le& Q+ K  \left( P+ 2 P \exp\Big\{ \frac{ T^3}{\la_1} \Big( MK+ NQ+NKP \Big)   \Big\}     \right).
\en
It follows that
\be
 \| G(v^\ep)(t) -G(u^\ep)(t)\| \le R  \Big\| v^\ep(t) - u^\ep(t) \Big\|
\en
where 
\be
R= M K+ NQ+ NK  \left( P+ 2 P \exp\Big\{ \frac{ T^3}{\la_1} \Big( MK+ NQ+NKP \Big)   \Big\}     \right).
\en
It follows from \eqref{ui} that
\be
\|v^\ep(t)-u^\ep(t)\|^2&=&\sum\limits_{p=1}^\infty \Big[ \coshe (\varphi^\ep_p-\varphi_p) +\frac{\sinhe}{\canla}(g^\ep_p-g_p)\nn\\
&+&\int\limits_{0}^{t}\frac{\sinhes}{\canla}
 \Big(G_{p}(v^{\ep})(s)- G_{p}(u^{\ep})(s)  \Big)  ds\Big]^2\nn\\
&\le&  \sum\limits_{p=1}^\infty \Big[3 |\coshe|^2(\varphi^\ep_p-\varphi_p)^2 +3\Big(   \frac{\sinhe}{\canla}   \Big)^2 (g^\ep_p-g_p)^2\nn\\
&\quad +&3t\int\limits_{0}^{t}\left( \frac{\sinhes}{\canla}
\right)^2 
 \Big(G_{p}(v^{\ep})(s)- G_{p}(u^{\ep})(s)  \Big)^2  ds\Big].
\en
 Using the Lipschitzian property of $f$, we get the following inequality
\be
\|v^\ep(t)-u^\ep(t)\|^2 &\le& 3 \ep^{-\frac{2t}{T}} \|\varphi^\ep-\varphi\|^2+ \frac{3}{{\la_1}} \ep^{-\frac{2t}{T}} \|g^\ep-g\|^2\nn\\
&~~~~~+&\frac{3t}  {{\la_1}} \int\limits_{0}^{t} \ep^{\frac{2s-2t}{T}} \| f(s,v^{\ep}(s))- f(s,u^{\ep}(s))    \|^2ds\nn\\
&\le&  3 \ep^{-\frac{2t}{T}}  \left(  \|\varphi^\ep-\varphi\|^2 +\frac{\|g^\ep-g\|^2  }{\la_1  }   \right)+\frac{3tR^2}  {{\la_1}} \int\limits_{0}^{t} \ep^{\frac{2s-2t}{T}} \| v^{\ep}(s)-u^{\ep}(s)\|^2ds.
\en
This implies that
\begin{align*}
\ep^{\frac{2t}{T}}\|v^\ep(t)-u^\ep(t)\|^2 
\le 3  \left(  \|\varphi^\ep-\varphi\|^2 +\frac{\|g^\ep-g\|^2  }{\la_1  }   \right)
+\frac{3tR^2}  {{\la_1}} \int\limits_{0}^{t} \ep^{\frac{2s}{T}} \| v^{\ep}(s)-u^{\ep}(s)    \|^2ds.
\end{align*}
Applying Gronwall's inequality, we have
\begin{align}
\ep^{\frac{2t}{T}}\|v^\ep(t)-u^\ep(t)\|^2 &\le   3 \exp\left\{ \frac{3TR^2t}  {{\la_1}} \right\}   \left(  \|\varphi^\ep-\varphi\|^2 +\frac{\|g^\ep-g\|^2  }{\la_1  }   \right) \nn\\
&\le 3\exp\left\{ \frac{3TR^2t}  {{\la_1}} \right\} (1+\frac{1}{\la_1}) \ep^2 .
\end{align}
By simplification, it yields
\bes
\|v^\ep(t)-u^\ep(t)\| \le \sqrt{ 3\exp\left\{ \frac{3TR^2t}  {{\la_1}} \right\} (1+\frac{1}{\la_1})  } \ep^{1-\frac{t}{T}}. \label{ss2}  
\ens
\end{proof}

Combining \eqref{ss1} and \eqref{ss2}, we thus obtain
\be
\|u^\ep(t)-u(t)\| &\le& \|v^\ep(t)-u^\ep(t)\|+ \|v^\ep(t)-u(t)\| \nn\\
&\le&  \sqrt{ 3\exp\left\{ \frac{3TR^2t}  {{\la_1}} \right\} (1+\frac{1}{\la_1})  } \ep^{1-\frac{t}{T}}+ 2 P \exp\Big\{ \frac{ T^2t}{\la_1} \Big( MK+ NQ+NKP \Big)   \Big\}   \ep^{1-\frac{t}{T}} 
\en
This completes the proof of Theorem.

\section{A numerical example}

\textcolor{red}{In this section, we provide an example}  in order to illustrate how the proposed regularized solution approximates the exact solution for nonlinear elliptic
problems. The example is involved with the operator $-\dfrac{\partial^{2}}{\partial x^{2}}$
and the domain $D\left(A\right)=H_{0}^{1}\left(0,1\right)\subset L^{2}\left(0,1\right)$. Then, the problem has the following form.

\begin{equation}
\begin{cases}
u_{tt}+u_{xx}=F\left(u\right)+G\left(x,t\right), & \left(x,t\right)\in\left(0,1\right)\times\left(0,1\right),\\
u\left(0,t\right)=u\left(1,t\right)=0, & t\in\left(0,1\right),\\
u\left(x,0\right)=\varphi\left(x\right),u_{t}\left(x,0\right)=g\left(x\right), & x\in\left(0,1\right),
\end{cases}
\end{equation}
where $F,G,\varphi$ and $g$ are given as follows.
\begin{equation}
F\left(u\right)=\frac{1}{a^{3}}u^{3},
\end{equation}
\begin{equation}
G\left(x,t\right)=2at\left(1-3x\right)-t^{3}x^{6}\left(1-x\right)^{3},
\end{equation}

\begin{equation}
\varphi\left(x\right)=0,\; g\left(x\right)=ax^{2}\left(1-x\right).
\end{equation}
This equation \textcolor{red}{would be considered as a} kind of  the Lane-Emden equations.  It is not too hard to see that the exact solution is $atx^{2}\left(1-x\right)$
where $a\in\mathbb{R}\backslash\left\{ 0\right\} $. An orthonormal
eigenbasis in $L^{2}\left(0,1\right)$ is $\phi_{p}\left(x\right)=\sqrt{2}\sin\left(\sqrt{\lambda_{p}}x\right)$
and $\lambda_{p}=p^{2}\pi^{2}$ is the corresponding eigenvalue. As
a result, choose $a=1$, we have 

\begin{equation}
u\left(x,t\right)=\sum_{p=1}^{\infty}\left[\cosh\left(\sqrt{\lambda_{p}}t\right)\varphi_{p}+\frac{\sinh\left(\sqrt{\lambda_{p}}t\right)}{\sqrt{\lambda_{p}}}g_{p}+\int_{0}^{t}\frac{\sinh\left(\sqrt{\lambda_{p}}\left(t-s\right)\right)}{\sqrt{\lambda_{p}}}f_{p}\left(u\right)\left(s\right)ds\right]\phi_{p}\left(x\right),
\end{equation}

where

\[
\varphi_{p}=\int_{0}^{1}\varphi\left(x\right)\phi_{p}\left(x\right)dx,\; g_{p}=\int_{0}^{1}g\left(x\right)\phi_{p}\left(x\right)dx,\; f_{p}\left(u\right)\left(s\right)=\int_{0}^{1}\left[F\left(u\right)+G\left(x,s\right)\right]\phi_{p}\left(x\right)dx.
\]

\begin{remark}
We shall approximate the regularized solution by taking perturbation
numbers in data function by two ways. The perturbation is intended
to define as $\epsilon\mbox{rand}\left(.\right)$ where each random
term rand$\left(.\right)$ will be determined on $\left[-1,1\right]$
uniformly, i.e.

\begin{eqnarray*}
f^{\epsilon}\left(.\right) & = & f\left(.\right)+\epsilon\mbox{rand}\left(.\right),\\
f^{\epsilon}\left(.\right) & = & f\left(.\right)\left(1+\frac{\epsilon\mbox{rand}\left(.\right)}{\left\Vert f\right\Vert }\right).
\end{eqnarray*}

In particular, we let

\[
\varphi^{\epsilon}\left(x\right)=\epsilon\mbox{rand}\left(.\right),\; g^{\epsilon}\left(x\right)=g\left(x\right)\left(1+\sqrt{105}\epsilon\mbox{rand}\left(.\right)\right).
\]

\end{remark}

\begin{remark}
The aim of the numerical experiments is to observe $\epsilon=10^{-r}$
where $r=\overline{1,10}$. The couple of $\left(\varphi^{\epsilon},g^{\epsilon}\right)$
plays as measured data with a random noise. Then, the regularized
solution is expected to be closed to the exact solution under
a proper discretization. 
\end{remark}
As we introduced, we proceed to define stability terms. Those are

\begin{equation}
\cosh^{\epsilon}\left(\sqrt{\lambda_{p}}t\right)=\frac{Q\left(\epsilon,\lambda_{p}\right)e^{\sqrt{\lambda_{p}}t}+e^{-\sqrt{\lambda_{p}}t}}{2},\label{QQ}
\end{equation}

\begin{equation}
\sinh^{\epsilon}\left(\sqrt{\lambda_{p}}t\right)=\frac{Q\left(\epsilon,\lambda_{p}\right)e^{\sqrt{\lambda_{p}}t}-e^{-\sqrt{\lambda_{p}}t}}{2},
\end{equation}

\begin{equation}
\sinh^{\epsilon}\left(\sqrt{\lambda_{p}}\left(t-s\right)\right)=\frac{Q\left(\epsilon,\lambda_{p}\right)e^{\sqrt{\lambda_{p}}\left(t-s\right)}-e^{-\sqrt{\lambda_{p}}\left(t-s\right)}}{2},
\end{equation}

\begin{equation}
Q\left(\epsilon,\lambda_{p}\right)=\frac{e^{-\sqrt{\lambda_{p}}}}{\epsilon+e^{-\sqrt{\lambda_{p}}}}.\label{QQQ}
\end{equation}

Therefore, we have the regularized solution.

\begin{equation}
v^{\epsilon}\left(x,t\right)=\sum_{p=1}^{\infty}\left[\cosh^{\epsilon}\left(\sqrt{\lambda_{p}}t\right)\varphi_{p}^{\epsilon}+\frac{\sinh^{\epsilon}\left(\sqrt{\lambda_{p}}t\right)}{\sqrt{\lambda_{p}}}g_{p}^{\epsilon}+\int_{0}^{t}\frac{\sinh^{\epsilon}\left(\sqrt{\lambda_{p}}\left(t-s\right)\right)}{\sqrt{\lambda_{p}}}f_{p}\left(v^{\epsilon}\right)\left(s\right)ds\right]\phi_{p}\left(x\right). \label{regu}
\end{equation}

After dividing the time \textcolor{red}{interval into equal subintervals} $t_{i}=i\Delta t,\Delta t=\dfrac{1}{M},i=\overline{0,M}$, \textcolor{red}{the nonlinear term in \eqref{regu}
\[
\int_{0}^{t_i}\frac{\sinh^{\epsilon}\left(\sqrt{\lambda_{p}}\left(t_i-s\right)\right)}{\sqrt{\lambda_{p}}}f_{p}\left(v^{\epsilon}\right)\left(s\right)ds 
= \sqrt{2}\int_{0}^{t_i}\frac{\sinh^{\epsilon}(p\pi(t_i-s))}{p\pi}\int_{0}^{1}(F(v^{\epsilon})(x,s)+G(x,s))\sin(p\pi x)dxds,
\]
can be approximated by an iterative scheme
\begin{equation}
\frac{\sqrt{2}}{p\pi}\sum_{j=1}^{i}\int_{t_{j-1}}^{t_{j}}\int_{0}^{1}\sinh^{\epsilon}\left(p\pi\left(t_{i}-s\right)\right)\left(F(v^{\epsilon})(x,t_{j-1})+G\left(x,s\right)\right)\sin\left(p\pi x\right)dxds. \label{app}
\end{equation}
Let $N$ be a cut-off constant which will be discussed later, then \eqref{regu} gives
\begin{eqnarray}
v_{N}^{\epsilon}\left(x,t_{i}\right) & = & v_{N,i}^{\epsilon}\left(x\right)\nonumber \\
 & = & w_{1,i}^{\epsilon}\sin\left(\pi x\right)+w_{2,i}^{\epsilon}\sin\left(2\pi x\right)+...+w_{N,i}^{\epsilon}\sin\left(N\pi x\right)\nonumber \\
 & = & \begin{bmatrix}w_{1,i}^{\epsilon} & w_{2,i}^{\epsilon} & \cdots & w_{N,i}^{\epsilon}\end{bmatrix}\begin{bmatrix}\sin\left(\pi x\right)\\
\sin\left(2\pi x\right)\\
\vdots\\
\sin\left(N\pi x\right)
\end{bmatrix},
\end{eqnarray}
where $w_{p,i}^{\epsilon}$ is defined by} 

\begin{eqnarray}
\frac{1}{2}w_{p,i}^{\epsilon} & = & \cosh^{\epsilon}\left(p\pi t_{i}\right)\int_{0}^{1}\varphi^{\epsilon}\left(x\right)\sin\left(p\pi x\right)dx+\frac{\sinh^{\epsilon}\left(p\pi t_{i}\right)}{p\pi}\int_{0}^{1}g^{\epsilon}\left(x\right)\sin\left(p\pi x\right)dx\nonumber \\
 &  & +\frac{1}{p\pi}\sum_{j=1}^{i}\int_{t_{j-1}}^{t_{j}}\int_{0}^{1}\sinh^{\epsilon}\left(p\pi\left(t_{i}-s\right)\right)\left(\left(v_{N}^{\epsilon}\left(x,t_{j-1}\right)\right)^{3}+G\left(x,s\right)\right)\sin\left(p\pi x\right)dxds,
 \label{wap}
\end{eqnarray}

\begin{equation}
v_{N}^{\epsilon}\left(x,t_{0}\right)=v_{N,0}^{\epsilon}\left(x\right)=\varphi^{\epsilon}\left(x\right).\label{guess}
\end{equation}

\textcolor{red}{Here we have introduced an iterative scheme of regularized solution \eqref{regu} followed by the time-step approximation \eqref{app} with initial guess \eqref{guess}. We also note that the first and second terms in the right hand side of \eqref{wap} come from $\varphi_p^{\epsilon}$ and $g_p^{\epsilon}$ and \eqref{QQ}-\eqref{QQQ}. Since $\sqrt{2}$ from the eigenbasis $\phi_p$ appears twice in \eqref{regu}, we put it into the left hand side of \eqref{wap} for the sake of simple computation.}

Finally, let $x_{j}=j\Delta x,\Delta x=\dfrac{1}{K},j=\overline{0,K}$,
\textcolor{red}{we obtain a fully discretization of regularized solution as follows:}

\begin{equation}
v_{N,i}^{\epsilon}\left(x_{j}\right)\equiv v_{N}^{\epsilon}\left(x_{j},t_{i}\right)=w_{1,i}^{\epsilon}\sin\left(\pi x_{j}\right)+w_{2,i}^{\epsilon}\sin\left(2\pi x_{j}\right)+...+w_{N,i}^{\epsilon}\sin\left(N\pi x_{j}\right).
\end{equation}

The whole process is concluded into four steps.

\textbf{Step 1.} Have $\epsilon$, choose $N=N_{0},K=K_{0}$ and $M=M_{0}$
respectively. We get

\begin{equation}
x_{j}=j\Delta x,\Delta x=\dfrac{1}{K},j=\overline{0,K},
\end{equation}

\begin{equation}
t_{i}=i\Delta t,\Delta t=\dfrac{1}{M},i=\overline{0,M}.
\end{equation}

\textbf{Step 2.} Put $v_{N}^{\epsilon}\left(x,t_{i}\right)=v_{N,i}^{\epsilon}\left(x\right),i=\overline{0,M}$
and $v_{N,0}^{\epsilon}\left(x\right)=\varphi^{\epsilon}\left(x\right)$.
We find out

\begin{equation}
V_{N}^{\epsilon}\left(x\right)=\left[v_{N,0}^{\epsilon}\left(x\right),v_{N,1}^{\epsilon}\left(x\right),...,v_{N,M}^{\epsilon}\left(x\right)\right]^{T}\in\mathbb{R}^{M+1}.
\end{equation}

\textbf{Step 3.} For $j=\overline{0,K},$ put $v_{N,i}^{\epsilon}\left(x_{j}\right)=v_{N,i,j}^{\epsilon}$,
we present

\begin{eqnarray}
U_{N,M,K}^{\epsilon} & = & \left[v_{N,0}^{\epsilon}\left(x_{j}\right),v_{N,1}^{\epsilon}\left(x_{j}\right),...,v_{N,M}^{\epsilon}\left(x_{j}\right)\right]\\
 & = & \begin{bmatrix}v_{N,0,0}^{\epsilon} & v_{N,0,1}^{\epsilon} & \cdots & v_{N,0,K}^{\epsilon}\\
v_{N,1,0}^{\epsilon} & v_{N,1,1}^{\epsilon} & \cdots & v_{N,1,K}^{\epsilon}\\
\vdots & \vdots & \ddots & \vdots\\
v_{N,M,0}^{\epsilon} & v_{N,M,1}^{\epsilon} & \ldots & v_{N,M,K}^{\epsilon}
\end{bmatrix}\in\mathbb{R}^{M+1}\times\mathbb{R}^{K+1}.
\end{eqnarray}

\textbf{Step 4.} Calculate the error

\begin{equation}
E_{N}^{\epsilon}\left(t_{i}\right)=\sqrt{\sum_{j=0}^{K}\left|v_{N}^{\epsilon}\left(x_{j},t_{i}\right)-u_{ex}\left(x_{j},t_{i}\right)\right|^{2}},\quad i=\overline{0,M}.
\end{equation}

Figure 2, Figure 3 and Figure
4 illustrate regularized solutions in 2D graph
in many cases which are known in each caption. On the other hand, in Figure 1
and Figure 5, these solutions and exact solution in 3D graph are represented. Although we consider
\textcolor{red}{several} values of noise in Table 1, showing 2D and 3D graphs is stopped
at a reasonable \textcolor{red}{cut-off number} by observation. We only show $r=\overline{1,4}$
for both 2D and 3D graphs.

In this example, we simply choose a slightly coarse grid $M=12,K=20$
because we want to reduce computational workloads. However, before
deciding to choose those, we make a test with a finer grid $M=16,K=30$
then even it is worse. In particular, for $\epsilon=10^{-4}$, $E_{N}^{\epsilon}\left(\dfrac{1}{2}\right)$
in the coarse grid is $1.3001169841\times10^{-2}$ while $1.5925046512\times10^{-2}$
is for the finer. Also, for $\epsilon=10^{-5}$, $E_{N}^{\epsilon}\left(\dfrac{1}{4}\right)$
in the coarse is $6.3529040819\times10^{-3}$ while $7.7873931681\times10^{-3}$
is for the finer one. 

We note that larger $N$ mostly leads to better approximation. For
example, we merely choose $N=2$, then \textcolor{red}{in} Figure
2-Figure 4, the regularized solution still do not fit the exact solution
completely. Thus, we consider two supplement cases $N=3$ and $N=4$
which are shown in the Figure 6-Figure 8 and Table 2. They are all
extremely better than what we obtain in the Table 1 and Figure 2-Figure
4. However, based on the advice in \cite{Ting}, we should choose $N=4$
or $N=5$ to not only get the whole desired, but also reduce computational
workloads.

~~\\\
\section{Conclusion}

In this paper, we propose a method of integral  equation  to solve the Cauchy problem
for elliptic equations with nonlinear source.  This problem may be difficult and there are few results on the regularized problem.  From that point, we aim to consider the regularization method for this problem in theoretical framework. The convergence results have been presented for the cases of $0 \le t < T $  and $t=T$ under some assumptions for the exact solution. However, our method still has a little theoretical range since the class of function $f$ is still
small. This makes the applicability of the method very narrow. Moreover, in the numerical result, there is an issue about choosing the truncation number which plays a role in regularization effect. In our future research, we will, therefore, consider the regularized problem in the
case where $f$ is a general locally Lipchitz function and study the theoretical analysis regarding the influence of the truncation term.\\

\section*{Acknowledgments}  This work is supported by Vietnam National University Ho Chi Minh City (VNU-HCM) under Grant No. B2014-18-01. \textcolor{red}{The third author, Vo Anh Khoa, would like to give many thanks to his old professor, Prof. Dang Duc Trong, for the whole-hearted guidance when the author was studying at Department of Mathematics and Computer Science, Ho Chi Minh City University of Science, Vietnam. Also, the authors desire to thank the handling editor and anonymous referees
for their most helpful comments on this paper.}

\newpage

\noindent \begin{center}
\begin{table}
\noindent \begin{centering}
\begin{tabular}{|c|c|c|c|}
\hline 
$\epsilon$ & $E_{N}^{\epsilon}\left(\dfrac{1}{2}\right)$ & $E_{N}^{\epsilon}\left(\dfrac{1}{4}\right)$ & $E_{N}^{\epsilon}\left(\dfrac{3}{4}\right)$\tabularnewline
\hline 
1.0E-01 & 3.6697975496E-01 & 2.6038073148E-01 & 6.3238102008E-01\tabularnewline
\hline 
1.0E-02 & 5.7702326175E-02 & 3.4823150118E-02 & 8.6516725190E-02\tabularnewline
\hline 
1.0E-03 & 2.1897073639E-02 & 1.2765356426E-02 & 3.2397554936E-02\tabularnewline
\hline 
1.0E-04 & 1.3001169841E-02 & 6.5571743949E-03 & 1.9598409706E-02\tabularnewline
\hline 
1.0E-05 & 1.2704207763E-02 & 6.3529040819E-03 & 1.9055362258E-02\tabularnewline
\hline 
\end{tabular}
\par\end{centering}

\caption{Errors between the regularized solution and exact solution for $t=\dfrac{1}{4};\dfrac{1}{2};\dfrac{3}{4}$.}
\end{table}

\par\end{center}

\begin{table}
\noindent \begin{centering}
\begin{tabular}{|c|c|c|c|}
\hline 
$\epsilon=10^{-4}$ & $E_{N}^{\epsilon}\left(\dfrac{1}{2}\right)$ & $E_{N}^{\epsilon}\left(\dfrac{1}{4}\right)$ & $E_{N}^{\epsilon}\left(\dfrac{3}{4}\right)$\tabularnewline
\hline 
$N=3$ & 1.0711862180E-02 & 5.5059970016E-03 & 1.5931956961E-02\tabularnewline
\hline 
$N=4$ & 3.8073451089E-03 & 1.9328671997E-03 & 5.6915418779E-03\tabularnewline
\hline 
\end{tabular}
\par\end{centering}

\caption{Errors between the regularized solution and exact solution at $t=\dfrac{1}{4};\dfrac{1}{2};\dfrac{3}{4}$
for $N=3;4$.}
\end{table}

\noindent \begin{center}
\begin{figure}
\noindent \begin{centering}
\includegraphics[scale=0.7]{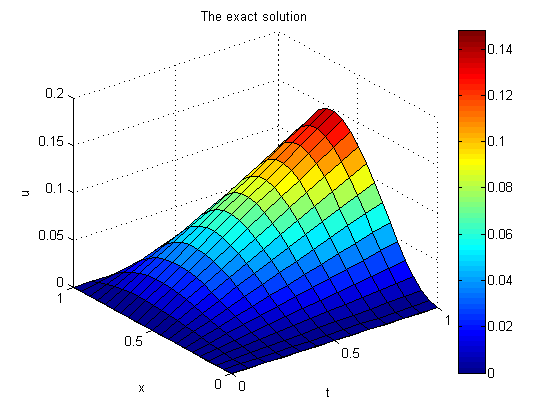}
\par\end{centering}

\caption{The exact solution $u_{ex}=tx^{2}\left(1-x\right)$.}
\end{figure}

\par\end{center}

\begin{figure}
\noindent \begin{centering}
\includegraphics[scale=0.43]{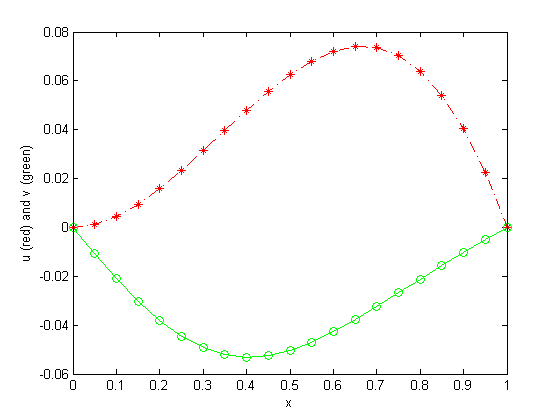}\includegraphics[scale=0.43]{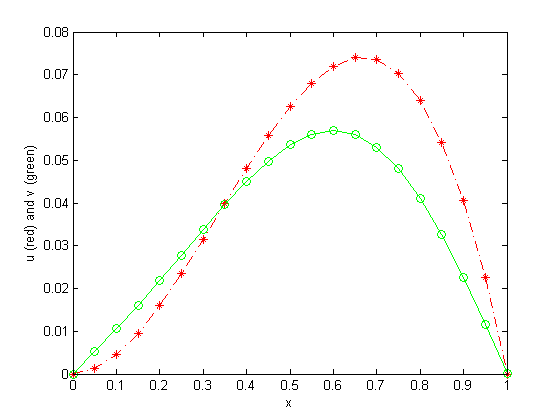}
\par\end{centering}

\noindent \begin{centering}
\includegraphics[scale=0.43]{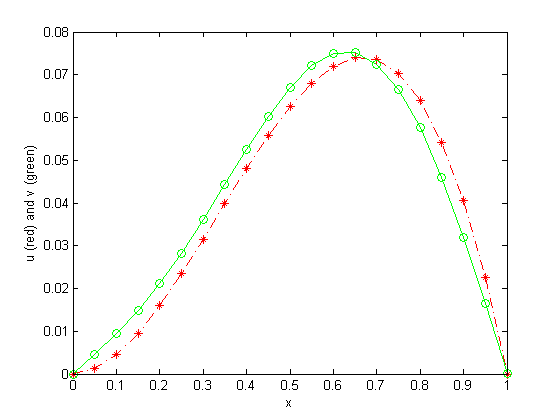}\includegraphics[scale=0.43]{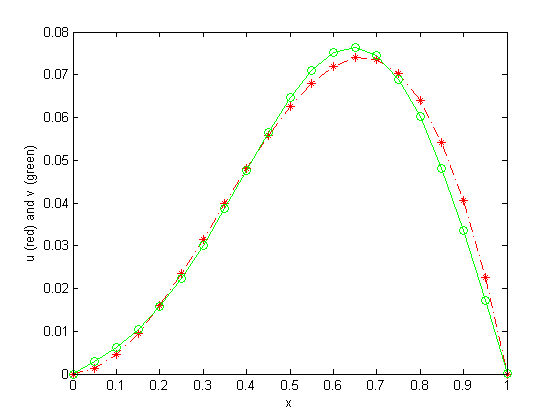}
\par\end{centering}

\caption{The regularized solution (green) and exact solution (red) at $t=\dfrac{1}{2}$
for $\epsilon=10^{-r}$ with $r=1;2;3;4$.}

\end{figure}

\begin{figure}
\noindent \begin{centering}
\includegraphics[scale=0.43]{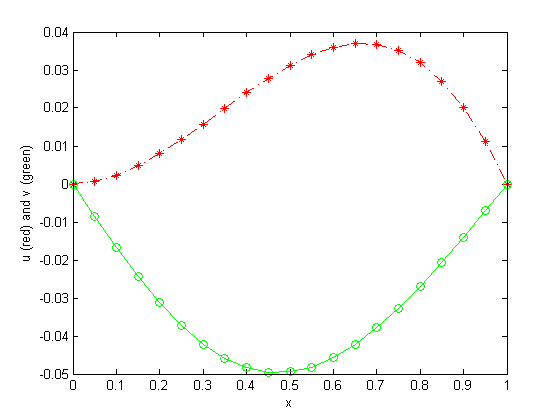}\includegraphics[scale=0.43]{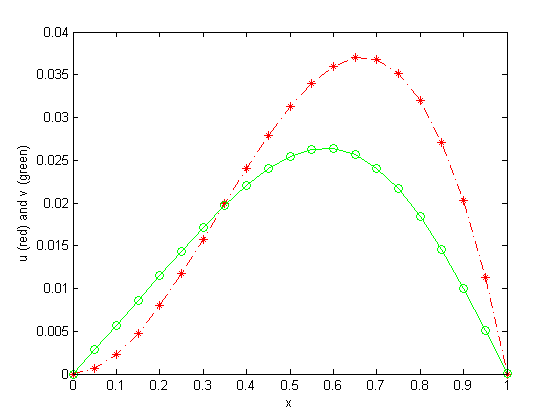}
\par\end{centering}

\noindent \begin{centering}
\includegraphics[scale=0.43]{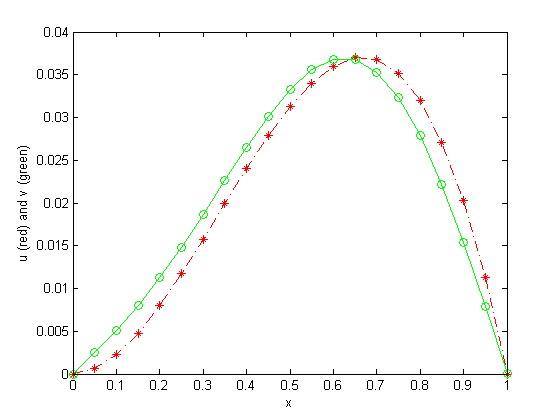}\includegraphics[scale=0.43]{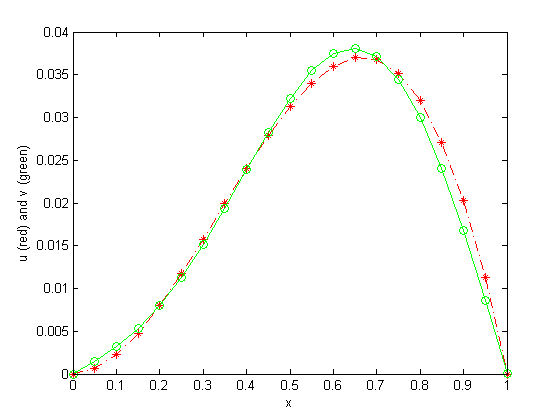}
\par\end{centering}

\caption{The regularized solution (green) and exact solution (red) at $t=\dfrac{1}{4}$
for $\epsilon=10^{-r}$ with $r=1;2;3;4$.}

\end{figure}

\begin{figure}
\noindent \begin{centering}
\includegraphics[scale=0.43]{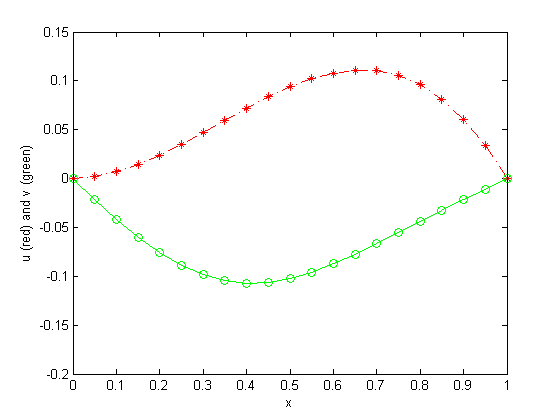}\includegraphics[scale=0.43]{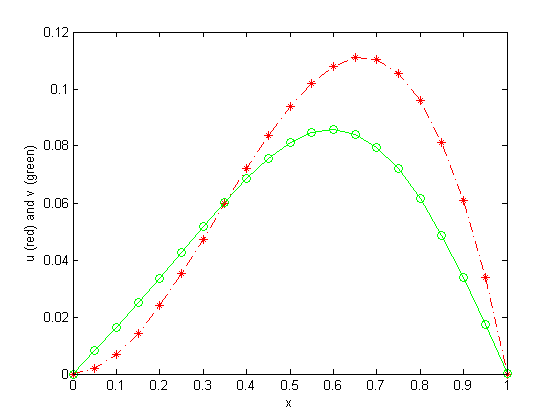}
\par\end{centering}

\noindent \begin{centering}
\includegraphics[scale=0.43]{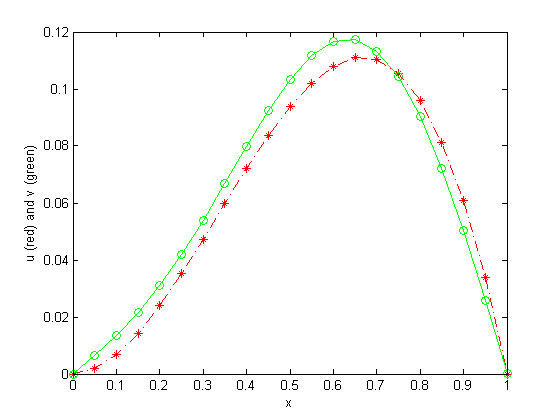}\includegraphics[scale=0.43]{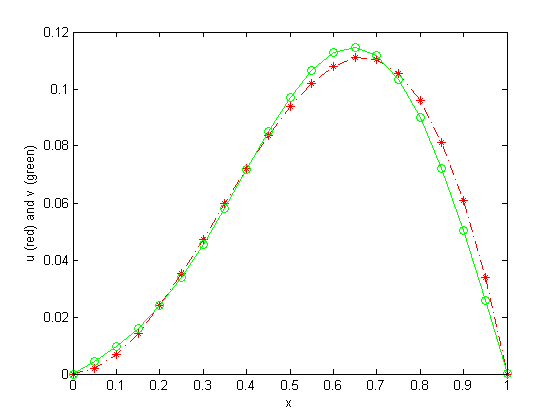}
\par\end{centering}

\caption{The regularized solution (green) and exact solution (red) at $t=\dfrac{3}{4}$
for $\epsilon=10^{-r}$ with $r=1;2;3;4$.}

\end{figure}

\begin{figure}
\noindent \begin{centering}
\includegraphics[scale=0.43]{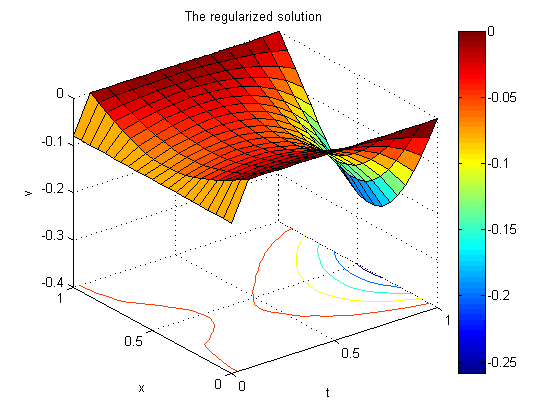}\includegraphics[scale=0.43]{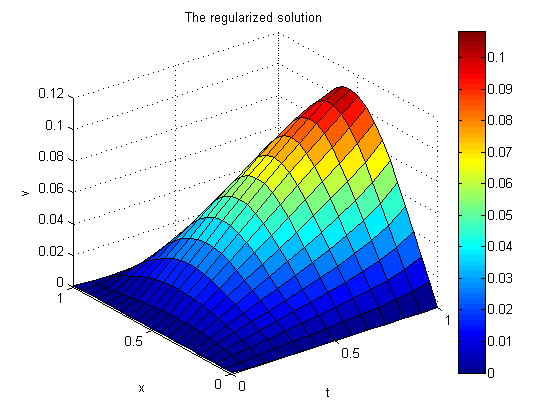}
\par\end{centering}

\noindent \begin{centering}
\includegraphics[scale=0.43]{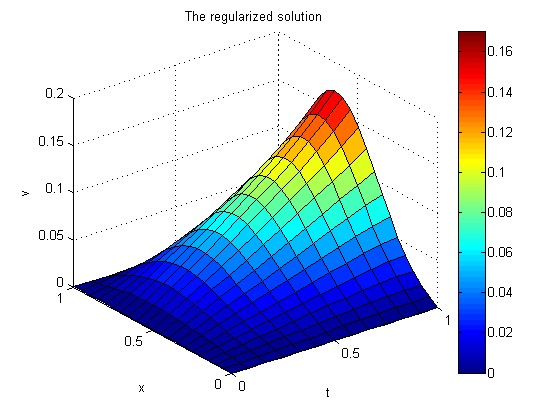}\includegraphics[scale=0.43]{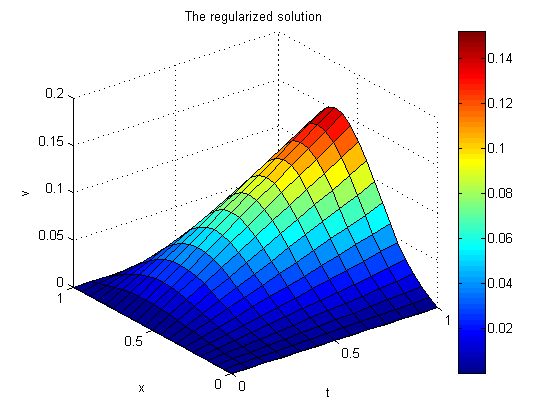}
\par\end{centering}

\caption{The regularized solution in 3D for $\epsilon=10^{-r}$ with $r=1;2;3;4$.}
\end{figure}

\begin{figure}
\noindent \begin{centering}
\includegraphics[scale=0.43]{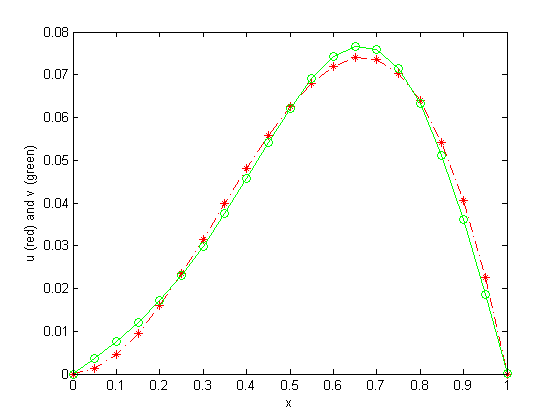}\includegraphics[scale=0.43]{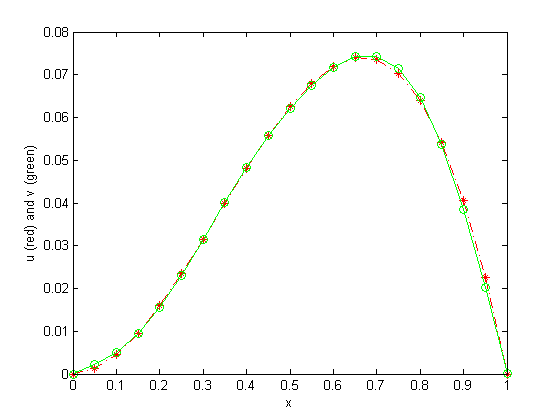}
\par\end{centering}

\caption{The regularized solution at $t=\dfrac{1}{2}$ for $N=3$ and $N=4$
with $\epsilon=10^{-4}$. }
\end{figure}

\begin{figure}
\noindent \begin{centering}
\includegraphics[scale=0.43]{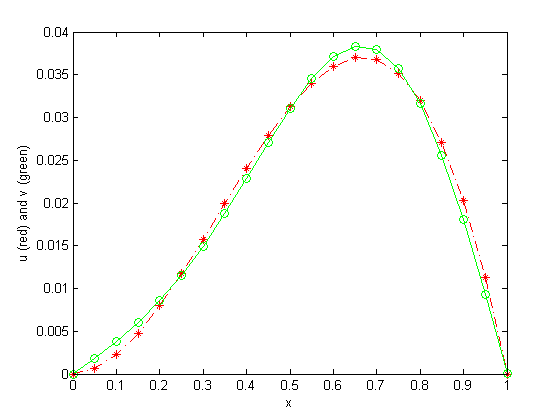}\includegraphics[scale=0.43]{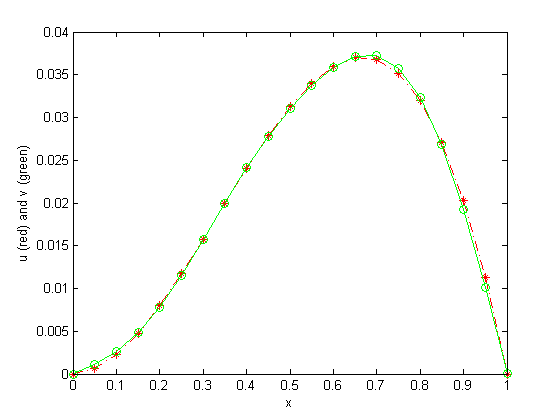}
\par\end{centering}

\caption{The regularized solution at $t=\dfrac{1}{4}$ for $N=3$ and $N=4$
with $\epsilon=10^{-4}$. }
\end{figure}

\begin{figure}
\noindent \begin{centering}
\includegraphics[scale=0.43]{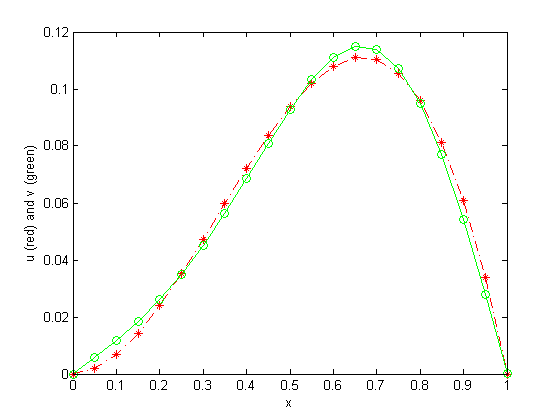}\includegraphics[scale=0.43]{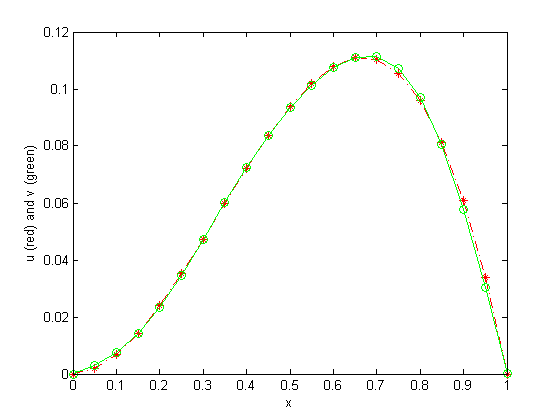}
\par\end{centering}

\caption{The regularized solution at $t=\dfrac{3}{4}$ for $N=3$ and $N=4$
with $\epsilon=10^{-4}$.}
\end{figure}

\end{document}